\documentclass[xcolor=dvipsnames,svgnames,table,reqno]{amsart}

\input xy
\xyoption{all}
\usepackage{accents}
\usepackage{epsfig}
\usepackage{xcolor}
\usepackage{amsthm}
\usepackage{amssymb}
\usepackage{bbm}
\usepackage{amsmath}
\usepackage{amscd}
\usepackage{amsopn}
\usepackage{graphicx}

\usepackage{tikz}
\usetikzlibrary{shapes,shadows,arrows}
\usepackage{tikz-cd}

\usepackage{xspace}

\usepackage{hhline}
\usepackage{easybmat}
\usepackage{caption}   
\usepackage{relsize}

\usepackage{url}
\usepackage{enumitem, hyperref}\hypersetup{colorlinks}

\usepackage{stmaryrd}

\usepackage[T1]{fontenc}

\colorlet{purpleB70}{blue!70!red}

\colorlet{orangeR65}{red!65!yellow}

\definecolor{red2}{HTML}{d41173}

\definecolor{neongreen}{HTML}{1bf702}

\definecolor{radicalred}{HTML}{FF355E}

\definecolor{denim}{HTML}{1560BD}

\definecolor{darkcyan}{rgb}{0.0, 0.55, 0.55}

\definecolor{cilek}{HTML}{FF43A4}

\definecolor{mor}{HTML}{9F00C5}

\definecolor{phlox}{rgb}{0.87, 0.0, 1.0}

\definecolor{fluorescentpink}{HTML}{FF1493}

\definecolor{napiergreen}{rgb}{0.16, 0.5, 0.0}

\definecolor{kellygreen}{rgb}{0.3, 0.73, 0.09}

\definecolor{parisgreen}{HTML}{ 50C878 }

\definecolor{palatinateblue}{rgb}{0.15, 0.23, 0.89}

\definecolor{ceruleanblue}{rgb}{0.16, 0.32, 0.75}

\definecolor{brandeisblue}{rgb}{0.0, 0.44, 1.0}

\definecolor{KLMblue}{HTML}{0FC0FC}

\definecolor{cinnamon}{rgb}{0.82, 0.41, 0.12}

\definecolor{darkorange}{rgb}{1.0, 0.55, 0.0}

\definecolor{darktangerine}{rgb}{1.0, 0.66, 0.07}

\definecolor{deepcarrotorange}{rgb}{0.91, 0.41, 0.17}

\definecolor{internationalorange}{HTML}{FF4F00}

\definecolor{persimmon}{HTML}{EC5800}

\definecolor{pumpkin}{HTML}{FF7518}

\definecolor{darkred}{rgb}{1,0,0} 
\definecolor{darkgreen}{rgb}{0,0.7,0}
\definecolor{darkblue}{rgb}{0,0,1}

\hypersetup{colorlinks,
linkcolor=darkblue,
filecolor=darkgreen,
urlcolor=darkred,
citecolor=darkgreen}

\makeatletter
\def\reflb#1#2{\begingroup
    #2%
    \def\@currentlabel{#2}%
    \phantomsection\label{#1}\endgroup
}
\makeatother

\numberwithin{equation}{section}
\newtheorem{Theorem}{Theorem}
\numberwithin{Theorem}{section}

\newtheorem{TheoremX}{Theorem}

\newtheorem{CorollaryX}[TheoremX]{Corollary}

\newtheorem   {Lemma}[Theorem]{Lemma}

\newtheorem   {Proposition}[Theorem]{Proposition}
\newtheorem   {Corollary}[Theorem]{Corollary}
\theoremstyle {definition}
\newtheorem   {Definition}[Theorem]{Definition}
\theoremstyle {remark}
\newtheorem   {Remark}[Theorem]{Remark}
\newtheorem   {Example}[Theorem]{Example}

\def    \eps    {\epsilon}

\newcommand{\CA}{{\mathcal A}}

\newcommand{\CS}{{\mathcal S}}

\newcommand{\supp}{\operatorname{supp}}

\newcommand{\id}{{\mathit id}}

\newcommand{\const}{{\mathit const}}

\newcommand{\charr}{{\mathit char}\,}

\newcommand{\fa}{{\mathfrak a}}
\newcommand{\fb}{{\mathfrak b}}

\newcommand{\ty}{\tilde{y}}

\newcommand{\tJ}{\tilde{J}}

\newcommand{\hgamma}{\hat{\gamma}}

\newcommand{\tx}{\tilde{x}}

\newcommand{\tz}{\tilde{z}}

\newcommand{\tH}{\tilde{H}}

\newcommand{\CB}{{\mathcal B}}

\def    \F      {{\mathbb F}}

\def    \R      {{\mathbb R}}

\def    \Z      {{\mathbb Z}}
\def    \N      {{\mathbb N}}

\def    \CP     {{\mathbb C}{\mathbb P}}

\def    \12     {{\frac{1}{2}}}

\def    \p      {\partial}

\def    \SH     {\operatorname{SH}}

\def    \HF     {\operatorname{HF}}

\def    \H      {\operatorname{H}}

\newcommand    \htop  {\operatorname{h_{\scriptscriptstyle{top}}}}

\newcommand   \slope {\operatorname{slope}}
\newcommand   \rmax {r_{\max}}
\newcommand   \rmin {r_{\min}}
\newcommand   \WW {\widehat{W}}
\newcommand   \hU {\widehat{U}}
\newcommand   \fbfl {\frak{b}^{\scriptstyle{Fl}}}

\newcommand   \fbdyn {\frak{b}^{\scriptstyle{Dyn}}}
\newcommand   \hbardyn {\hbar^{\scriptstyle{Dyn}}}

\begin{document}








\title [Barcode Entropy of Reeb Flows]{On the Barcode Entropy of Reeb Flows}

\author[Erman \c C\. inel\. i]{Erman \c C\. inel\. i}
\author[Viktor Ginzburg]{Viktor L. Ginzburg}
\author[Ba\c sak G\"urel]{Ba\c sak Z. G\"urel}
\author[Marco Mazzucchelli]{Marco Mazzucchelli}


\address{E\c C: ETH Z\"urich, R\"amistrasse 101, 8092 Z\"urich,
  Switzerland} \email{erman.cineli@math.eth.ch}

\address{VG: Department of Mathematics, UC Santa Cruz, Santa
  Cruz, CA 95064, USA} \email{ginzburg@ucsc.edu}

\address{BG: Department of
  Mathematics, 
  UCF, Orlando, FL 32816, USA} \email{basak.gurel@ucf.edu}

\address{MM: CNRS, UMPA, \'Ecole Normale Sup\'erieure de Lyon, 69364
  Lyon, France} \email{marco.mazzucchelli@ens-lyon.fr}



\date{\today} 

\thanks{The work is partially supported by the NSF grants DMS-2304207
  (BG) and DMS-2304206 (VG), Simons Foundation grants 855299 (BG) and
  MP-TSM-00002529 (VG), ERC Starting Grant 851701 via a postdoctoral
  fellowship (E\c{C}), and the ANR grants CoSyDy, ANR-CE40-0014 and
  COSY, ANR-21-CE40-0002 (MM)}

\begin{abstract}
  In this paper we continue investigating connections between Floer
  theory and dynamics of Hamiltonian systems, focusing on the barcode
  entropy of Reeb flows. Barcode entropy is the exponential growth
  rate of the number of not-too-short bars in the Floer or symplectic
  homology persistence module. The key novel result is that the
  barcode entropy is bounded from below by the topological entropy of
  any hyperbolic invariant set. This, combined with the fact that the
  topological entropy bounds the barcode entropy from above,
  established by Fender, Lee and Sohn, implies that in dimension three
  the two types of entropy agree. The main new ingredient of the proof
  is a variant of the Crossing Energy Theorem for Reeb flows.
\end{abstract}

\maketitle

\vspace{-0.2in}


\tableofcontents

\section{Introduction and main results}
\label{sec:intro+results}

\subsection{Introduction}
\label{sec:intro}
In this paper we continue investigating connections between Floer
theory and dynamics of Hamiltonian systems, focusing on the relation
of barcode entropy to topological entropy for Reeb flows.

Barcode entropy is an invariant associated with the sequence of Floer
homology persistence modules for the iterates of a Hamiltonian
diffeomorphism or the symplectic homology persistence module for a
Reeb flow. In both cases it measures the exponential growth rate of
the number of not-so-short bars in the barcode. It is closely related
to the topological entropy of the underlying Hamiltonian system.

For compactly supported Hamiltonian diffeomorphisms
$\varphi\colon M\to M$, barcode entropy was introduced in
\cite{CGG:Entropy}. We showed there that the barcode entropy
$\hbar(\varphi)$ of $\varphi$ is bounded from above by the topological
entropy $\htop(\varphi)$ and conversely that
$\htop(\varphi|_K)\leq \hbar(\varphi)$ whenever $K$ is a (locally
maximal) compact hyperbolic invariant set of $\varphi$. As a
consequence, $\hbar(\varphi)=\htop(\varphi)$ when $M$ is a surface by
the results of Katok, \cite{Ka}. (However, this equality does not hold
in general in higher dimensions, \cite{Ci:counterexample}.) For
geodesic flows, barcode entropy was defined in \cite{GGM} where
similar inequalities were established. For the Reeb flow
$\varphi^t=\varphi^t_\alpha$ of a contact form $\alpha$ on the
boundary $M$ of a Liouville domain, the barcode entropy
$\hbar(\alpha)$ was introduced in \cite{FLS}. There, a contact version
of the first inequality was also proved:
$\hbar(\alpha)\leq \htop(\varphi^t)$. Here we establish an analogue of
the second inequality: $\htop(\varphi|_K)\leq \hbar(\varphi)$, where
again $K$ is a locally maximal compact hyperbolic set, and hence
$\hbar(\alpha)=\htop(\varphi^t)$ when $\dim M=3$ by the results from
\cite{LY, LS} extending Katok's work to flows on three-manifolds. In
particular, as in the Hamiltonian case, barcode entropy can and does
detect topological entropy coming from localized sources such as a
hyperbolic set contained, for maps, in a small ball or, for flows, in
a thin mapping torus.

There is, of course, an immense body of work connecting dynamics,
e.g., topological entropy, of a Hamiltonian system, broadly
understood, with features of the underlying variational principle,
e.g., Morse or Floer homology. What distinguishes our approach to the
question is that it does not rely on unconditional global (symplectic)
topological properties of the map coming from, say, the exponential
growth of the Floer or Morse homology, in turn, determined by the
(symplectic) isotopy class of the map or topology of the
phase/configuration space.

In the setting of compactly supported Hamiltonian diffeomorphisms of
symplectic manifolds, this comes for granted: such maps are
(Hamiltonian) isotopic to the identity and the Floer homology is
independent of the map. Hence there cannot be any global topological
or Floer homological growth. However, in the setting of geodesic or
Reeb flows there are more possibilities. For instance, the Morse or
Floer homology can grow exponentially fast, forcing the flow to have
positive topological entropy. This growth is also captured by barcode
entropy but the phenomenon has been quite well understood
independently of this concept.

For instance, connections between topological entropy of a geodesic
flow and topology of the underlying manifold have been studied in,
e.g., \cite{Di, Ka82, Pa}. A variety of generalizations of these
classical results to Reeb flows have been obtained in \cite{AASS,
  Al16, Al19, ACH, AM, ADMM, AP, MS}, relating topological entropy of
Reeb flows to their Floer theoretic invariants (e.g., symplectic or
contact homology). Again, the global contact topology of the
underlying manifold is central to these results but plays little role
in our approach.

The key technical ingredient of the proof of our main theorem is the
Crossing Energy Theorem for Reeb flows, roughly speaking asserting
that for an admissible Hamiltonian $H$, any Floer trajectory
asymptotic to a periodic orbit of its Hamiltonian flow $\varphi_H^t$
corresponding to a periodic orbit of $\varphi_\alpha^t |_K$ has energy
bounded from below by some constant $\sigma>0$ independent of the
orbit and its period. (An admissible Hamiltonian is specifically
tailored for recasting Reeb dynamics in Hamiltonian terms; in
particular, it is autonomous and its flow on every positive level is a
reparametrization of the Reeb flow.)  For locally maximal hyperbolic
sets of Hamiltonian diffeomorphisms, several variants of this theorem
were originally proved in \cite{CGG:Entropy, GG:hyperbolic, GG:PR}. At
the center of the argument is the observation that all circles
$u(s,\,\cdot)$ in a low energy Floer cylinder $u$ for an iterated
Hamiltonian diffeomorphism are $\eps$-pseudo-orbits with $\eps>0$
independent of the period. The proofs, however, do not directly
translate to admissible Hamiltonians. Indeed, one of the main
difficulties arising in the contact setting is that the invariant set
of $\varphi_H^t$ corresponding to a locally maximal hyperbolic set of
$\varphi^t_\alpha$ is neither hyperbolic nor locally maximal, while
both conditions are essential.

A proof of the Crossing Energy Theorem in the Hamiltonian setting for
$\CP^n$ using generating function was given in \cite{Al} and its
counterpart for geodesic flows, also using finite-dimensional
reduction, can be found in \cite{GGM}. In \cite{CGGM} the theorem was
proved for an isolated hyperbolic periodic orbit of a Reeb flow in a
setting fairly close to the one adopted here. Finally, a variant of
the theorem for holomorphic curves in the symplectization is
established in \cite{CGP} and a version with Lagrangian boundary
conditions is proved in \cite{Me:new}.

All known to date lower bounds on barcode entropy type invariants rely
on a combination of hyperbolicity and crossing energy bounds. However,
crossing energy theorems have other applications outside the subject
of barcode entropy, ranging from multiplicity results for periodic
orbits in a variety of settings to Le Calvez--Yoccoz type theorems to
lower bounds on the spectral norm; see \cite{Ba1, Ba2, CGG:Spectral,
  CGGM, GG:hyperbolic, GG:nc, GG:PR}.

\subsection{Main definitions and results}
\label{sec:results}

Let $(W, d\alpha)$ be a Liouville domain. We will also use the
notation $\alpha$ for the contact form $\alpha|_M$ on the boundary
$M=\p W$.

Fix a ground field $\F$, which we suppress in the notation, and denote
the (non-equivariant) filtered symplectic homology of $W$ over $\F$
for the action interval $[0,\, \tau)$ by
$\SH^\tau(\alpha)$. Throughout this paper, the grading of symplectic
homology plays no role and we view $\SH^a(\alpha)$ as an ungraded
vector space over $\F$.  We make no assumptions on the first Chern
class $c_1(TW)$.

Together with the natural maps
$\SH^{\tau_0}(\alpha)\to \SH^{\tau_1}(\alpha)$ for
$\tau_0\leq \tau_1$, the symplectic homology forms a persistence
module; see Sections \ref{sec:persistence} and \ref{sec:SH}.  For
$\eps>0$, we denote by $\fb_\eps(\alpha,\tau)$ or just
$\fb_\eps(\tau)$, when $\alpha$ is clear from the context, the number
of bars of length greater than $\eps$, beginning in the range
$[0,\, \tau)$ in the barcode $\CB(\alpha)$ of this persistence
module. This is an increasing function in $\tau$ and $1/\eps$, locally
constant as a function of $\tau$ in the complement to
$\CS(\alpha)\cup \{0\}$, where $\CS(\alpha)$ is the action spectrum of
$\alpha$.

The barcode entropy of $\alpha$, denoted by $\hbar(\alpha)$, measures
the exponential growth rate of $\fb_\eps(\tau)$ and is defined as
follows.

\begin{Definition}
  \label{def:barcode_entropy}
The \emph{$\eps$-barcode entropy} of $\alpha$ is 
\begin{equation}
  \label{eq:eps-entropy}
  \hbar_\eps(\alpha): =
  \limsup_{\tau\to\infty}\frac{\log^+ \fb_\eps(\tau)}{\tau},
\end{equation} 
where $\log$ is taken base 2, $\log 0=-\infty$ and
$\log^+:=\max\{0,\log\}$, and the \emph{barcode entropy} of $\alpha$
is
\begin{equation}
  \label{eq:entropy}
  \hbar(\alpha):=\lim_{\eps \to 0^+}\hbar_\eps(\alpha)\in [0,\, \infty].
\end{equation}
\end{Definition}
A few comments are due at this point. First of all, since
$\fb_\eps(\tau)$ is an increasing function of $\tau$, we can replace
the upper limit in \eqref{eq:eps-entropy} as $\tau\to \infty$ by the
upper limit over an increasing sequence $\tau_i\to\infty$ as long as
this sequence is not too sparse. Namely, it is not hard to see that
$$
\hbar_\eps(\alpha)
= \limsup_{i\to\infty}\frac{\log^+ b_\eps(\tau_i)}{\tau_i}
$$
whenever $\tau_i\to\infty$ and $\tau_{i+1}/\tau_i\to 1$; cf.\
Proposition \ref{prop:hbar-dyn-V} and Theorem
\ref{thm:two-defs}. Second, $\hbar_\eps(\alpha)$ increases as
$\eps \to 0^+$, and hence the limit in \eqref{eq:entropy} does
exist. Clearly,
$$
\hbar_\eps(\alpha)\leq \hbar(\alpha),
$$
and, as is easy to see, for any $a>0$,
$$
\hbar_\eps(a\alpha)=a^{-1}\hbar_\eps(\alpha)
\textrm{ and }
\hbar(a\alpha)= a^{-1}\hbar(\alpha).
$$

\begin{Remark}
  Hypothetically, $\hbar_\eps(\alpha)$ and $\hbar(\alpha)$ might
  depend on the entire Liouville domain $(W,d\alpha)$ rather than just
  the contact form $\alpha$ on $M=\p W$. However, we are not aware of
  any examples where this happens. (Corollary \ref{cor:C} below
  implies, in particular, that $\hbar(\alpha)$ is completely
  determined by $(M,\alpha)$ when $\dim M=3$.) Moreover, according to
  \cite[Sect.\ 4.2]{FLS}, $\hbar(\alpha)$ is independent of the
  filling at least when $c_1(TW)=0$ and we tend to think in general as
  long as $(M,\alpha)$ is the boundary of the Liouville domain and
  hence $\hbar(\alpha)$ is defined.
\end{Remark}

Denote by $\htop(\alpha)$ the topological entropy of the Reeb flow
$\varphi^t_\alpha$ of $\alpha$. The next three theorems relating
barcode entropy and topological entropy are contact counterparts of
similar results for Hamiltonian diffeomorphisms and geodesic flows:
see \cite{CGG:Entropy} and \cite{GGM}.

\begin{TheoremX}[\cite{FLS}]
  \label{thm:A}
  For any Liouville domain $(W,d\alpha)$, we have
  $$
  \hbar(\alpha)\leq \htop(\alpha).
  $$
\end{TheoremX}

This theorem is originally proved in \cite{FLS}. (Strictly speaking,
the theorem is stated there under the additional conditions that
$c_1(TW)=0$. However, this condition appears to be immaterial for the
argument.) We will comment on the proof in Remark \ref{rmk:ThmA}. Note
that $\hbar(\alpha)<\infty$ by Theorem \ref{thm:A}.

The next theorem, which is the main new result of this paper, shows
that the barcode entropy can be positive and is related to the
hyperbolic invariant sets of $\varphi^t_\alpha$. We refer the reader
to, e.g., \cite{KH} for the definition of such sets and Section
\ref{sec:cross-energy+pf} for a further discussion.

\begin{TheoremX}
  \label{thm:B}
  Let $K$ be a compact hyperbolic invariant set of the Reeb
  flow $\varphi^t_\alpha$. Then
  $$
  \hbar(\alpha)\geq \htop(K),
  $$
  where we set $\htop(K):=\htop(\varphi^t_\alpha|_K)$.
\end{TheoremX}  

Next, recall that when $\dim M=3$,
$$
\htop(\alpha)=\sup_K \htop(K), 
$$
where $K$ ranges over all hyperbolic invariant sets, \cite{LY,
  LS}. (This is a generalization to flows on three-manifolds of a
theorem originally proved in \cite{Ka} for diffeomorphisms of
surfaces.)  Combining this fact with Theorems \ref{thm:A} and
\ref{thm:B}, we obtain the following.

\begin{CorollaryX}
  \label{cor:C}
  Assume that $\dim M=3$. Then
  $
  \hbar(\alpha)= \htop(\alpha).
  $
\end{CorollaryX}

This corollary is a Reeb counterpart of a similar result for
Hamiltonian diffeomorphisms of surfaces; see \cite[Thm.\
C]{CGG:Entropy}. The latter theorem does not generalize to higher
dimensions as the counterexamples constructed in
\cite{Ci:counterexample} show. While this construction does not
readily extend to Reeb flows, we do not expect Corollary \ref{cor:C}
to hold in higher dimensions either.

Furthermore, in dimension three $\htop(\alpha)$ is $C^0$
lower-semicontinuous in $\alpha$ at a $C^1$-open and dense set in the
space of all contact forms on $M$ as is proved in \cite{ADMP}; see
also \cite{ADMM}. By Corollary \ref{cor:C}, the same is true for
$\hbar(\alpha)$ whenever $M$ bounds a Liouville domain.

\begin{Remark}[Other types of barcode entropy]
  The barcode entropy in the Hamiltonian setting and for geodesic
  flows also has a relative counterpart associated with the filtered
  Floer or Morse homology for Lagrangian or geodesic chords,
  \cite{CGG:Entropy} and \cite[Sec.\ 4.4]{GGM}. A variant of relative
  barcode entropy for Reeb flows can also be defined via wrapped Floer
  homology and analogues of Theorems \ref{thm:A} and \ref{thm:B} have
  been recently proved in this setting, \cite{Fe24, Fe25}.

  Furthermore, yet two different versions of barcode entropy, both
  introduced in \cite{CGG:Growth} in the Hamiltonian framework, can
  also be defined in the Reeb setting and again we expect the above
  results to hold for them. These are sequential barcode entropy and
  total persistence entropy. Both entropies have properties similar to
  barcode entropy and, by construction, bound the ordinary barcode
  entropy from above. Hence, in both cases, Theorem \ref{thm:B}
  follows from its counterpart for the ordinary barcode entropy. In
  the Hamiltonian setting, Theorem \ref{thm:A} for these variants of
  barcode entropy was proved in \cite{CGG:Growth}.  We conjecture that
  these refinements of Theorem \ref{thm:A}, and hence Corollary
  \ref{cor:C}, also hold in the contact setting.

  For the sake of brevity we do not consider here these
  generalizations or modifications of barcode entropy, focusing
  instead on the proof of Theorem \ref{thm:B}.
\end{Remark}  

\begin{Remark}[Barcode entropy for geodesic flows]
  As we have pointed out in the introduction, barcode entropy for
  geodesic flows was defined in \cite{GGM}, where the analogues of
  Theorems \ref{thm:A} and \ref{thm:B} and Corollary \ref{cor:C} were
  also proved. In that specific case, the barcode entropy is equal to
  the barcode entropy considered here. This is essentially a
  consequence of the equality of the filtered Morse homology and the
  symplectic homology, although some attention needs to be paid to the
  definition of the latter; see \cite{AS, SW, Vi}, \cite[Rmk.\
  4.5]{GGM} and also \cite[Rmk.\ 4.10]{FLS}.
\end{Remark}

In general, the barcode growth reflects various aspects of dynamics
ranging from topological entropy as examined, for instance, in this
paper to complete integrability; see \cite{BG}.

The paper is organized as follows. In Section \ref{sec:prelim} we set
our conventions and notation and also discuss the class of
(semi-)admissible Hamiltonians used throughout the paper. The relevant
facts from Floer theory are assembled in Section \ref{sec:Floer}. In
Section \ref{sec:def-rev} we revisit the definition of barcode
entropy, and reformulate it in a way more suitable for dynamics
applications and prove equivalence of the definitions.  We derive
Theorem \ref{thm:B} from the Crossing Energy Theorem in Section
\ref{sec:cross-energy+pf}, which is then proved in
Section~\ref{sec:cross}.


\section{Preliminaries, conventions, and notation}
\label{sec:prelim}

\subsection{Persistence modules}
\label{sec:persistence}
Persistence modules play a central role in the definition of barcode
entropy. In this section we define the class of persistence modules
suitable for our goals and briefly touch upon their properties. We
refer the reader to \cite{PRSZ} for a general introduction to
persistence modules, their applications to geometry and analysis and
further references, although the class of modules they consider is
somewhat more narrow than the one we deal with here, and also to
\cite{BV, CB} for some of the more general results.

Fix a field $\F$ which we will suppress in the notation. Recall that a
\emph{persistence module} $(V,\pi)$ is a family of vector spaces $V_s$
over $\F$ parametrized by $s\in \R$ together with a functorial family
$\pi$ of structure maps. These are linear maps
$\pi_{st}\colon V_s\to V_t$, where $s\leq t$ and functoriality is
understood as that $\pi_{sr}=\pi_{tr}\pi_{st}$ whenever
$s\leq t\leq r$ and $\pi_{ss}=\id$. In what follows we often suppress
$\pi$ in the notation and simply refer to $(V,\pi)$ as $V$.  In such a
general form the concept is not particularly useful and usually one
imposes additional conditions on the spaces $V_t$ and the structure
maps $\pi_{st}$. These conditions vary depending on the context. Below
we spell out the framework most suitable for our purposes.

Namely, we require that there is a closed, bounded from below, nowhere
dense subset $\CS\subset \R$, which is called the \emph{spectrum} of $V$,
and the following four conditions are met:
\begin{itemize}

\item[\reflb{PM1}{\rm{(i)}}] The persistence module $V$ is
  \emph{locally constant} outside $\CS$, i.e., $\pi_{st}$ is an
  isomorphism whenever $s\leq t$ are in the same connected component
  of $\R\setminus \CS$.

\item[\reflb{PM2}{\rm{(ii)}}] The persistence module $V$ is
  \emph{$q$-tame}: $\pi_{st}$ has finite rank for all $s<t$.
 
\item[\reflb{PM3}{\rm{(iii)}}] \emph{Left-semicontinuity}: For all
  $t\in\R$,
  \begin{equation}
    \label{eq:semi-cont}
    V_t=\varinjlim_{s<t} V_s.
  \end{equation}

\item[\reflb{PM4}{\rm{(iv)}}] \emph{Lower bound}: $V_s=0$ when $s<s_0$
  for some $s_0\in\R$. (Throughout the paper we will assume that
  $s_0=0$.)

\end{itemize}

A few comments on this definition are in order. First, note that as a
consequence of \ref{PM1} and \ref{PM2}, $V_s$ is finite-dimensional
and \ref{PM3} is automatically satisfied when $s\not\in
\CS$. Furthermore, in several instances which are of interest to us,
$V_s$ is naturally defined only for $s\not\in\R$; then the definition
is extended to all $s\in\R$ by \eqref{eq:semi-cont}. By \ref{PM4}, we
can always assume that $s_0\leq \inf \CS$, i.e., $\CS$ is bounded from
below. We emphasize, however, that $\CS$ is not assumed to be bounded
from above and it is actually not in many examples we are interested
in. In what follows it will sometimes be convenient to include
$s=\infty$ by setting
$$
V_\infty=\varinjlim_{s\to\infty} V_s.
$$
Finally, in all examples we encounter here $\CS$ has zero measure and,
in fact, zero Hausdorff dimension, but this fact is never used in the
paper.

A basic example motivating requirements \ref{PM1}--\ref{PM4} is that
of the sublevel homology of a smooth function.

\begin{Example}[Homology of sublevels] Let $M$ be a smooth manifold
  and $f\colon M\to \R$ be a proper smooth function bounded from
  below. Set $V_s:=H_*\big(\{f<s\};\F\big)$ with the structure maps
  induced by inclusions. No other requirements are imposed on $f$,
  e.g., $f$ need not be Morse. However, it is not hard to see that
  conditions \ref{PM1}--\ref{PM4} are met with $\CS$ being the set of
  critical values of $f$. We note that one can have $\dim V_s=\infty$
  for $s\in\CS$ already when $M=S^1$, unless $f$ meets some additional
  conditions on $f$, e.g., that $f$ is real analytic or the critical
  points of $f$ are isolated.
\end{Example}

Recall furthermore that an \emph{interval persistence module}
$\F_{(a,\,b]}$, where $-\infty<a<b\leq \infty$, is defined by setting
$$
V_s:=\begin{cases}
  \F & \text{ when } s\in (a,\,b],\\
  0 & \text{ when } s \not\in (a,\,b],
\end{cases}
$$
and $\pi_{st}=\id$ if $a<s\leq t\leq b$ and $\pi_{st}=0$
otherwise. Interval modules are examples of simple persistence
modules, i.e., persistence modules that cannot be decomposed as a
(non-trivial) direct sum of other persistence modules.

A key fact that we will use in the paper is the normal form or
structure theorem asserting that every persistence module meeting the
above conditions can be decomposed as a direct sum of a countable
collection (i.e., a countable multiset) of interval persistence
modules. Moreover, this decomposition is unique up to re-ordering of
the sum. (In fact, conditions \ref{PM1}--\ref{PM4} are far from
optimal and can be considerably relaxed.) We refer the reader
\cite[Thm.\ 3.8]{BV} for a proof of this theorem for the class of
persistence modules considered here and further references, and also,
e.g., to \cite{CZCG, CB, ZC} for previous or related results.

This multiset $\CB(V)$ of intervals entering this decomposition is
referred to as the \emph{barcode} of $V$ and the intervals occurring
in $\CB(V)$ as \emph{bars}. For $\eps>0$ we denote by $\fb_\eps(V,s)$
or just $\fb_\eps(s)$ the number of bars $(a,\,b]$ in $\CB(V)$ with
$a<s$ of length $b-a> \eps$, counted with multiplicity. This is the
only numerical invariant of persistence modules used in this paper.
It is not hard to show that $\fb_\eps(s)<\infty$ for all $\eps>0$ and
$s<\infty$ under our conditions on $V$ (see Remark \ref{rmk:approx}
below), even though the total number of bars beginning below $s$ can
be infinite.

\begin{Remark}[Locally finite approximations]
  \label{rmk:approx}
  Every persistence module $V$ as above can be approximated with
  respect to the interleaving topology (see, e.g., \cite[Sec.\
    4.1]{BG}) by locally finite persistence modules $V'$.  The
  construction of $V'$ amounts to throwing away short bars and then
  adjusting $V_s$ for $s\in\CS$. Alternatively, in our case one can
  simply perturb the Hamiltonians or contact forms to ensure
  non-degeneracy. Then the truncations of $\CB(V)$ and $\CB(V')$
  are close with respect to the bottleneck distance, and for any
  $\delta>0$
  $$
  \fb_{\eps-\delta}(V',s)\geq \fb_{\eps}(V,s)\geq
  \fb_{\eps+\delta}(V',s)
  $$
  when $V'$ is sufficiently close to $V$.  Utilizing this fact, we
  could have worked with a more narrow class of locally finite
  persistence modules and used small perturbations to define
  $\fb_\eps(V,s)$. This is essentially the approach taken in
  \cite{CGG:Entropy}. However, here we find working with a broader
  class of persistence modules more convenient.
\end{Remark}

\subsection{Semi-admissible Hamiltonians, periodic orbits and the
  action functional}
\label{sec:setting}
In this section we spell out our conventions and notation on the
symplectic dynamics side, which are essentially identical to the ones
used in \cite{CGGM, GG:LS}, and also recall several elementary
properties of (semi-)admissible Hamiltonians to be used later.

Let, as in Section \ref{sec:results}, $\alpha$ be the contact form on
the boundary $M=\p W$ of a Liouville domain $W^{2n\geq 4}$. We will
also use the same notation $\alpha$ for a primitive of the symplectic
form $\omega$ on $W$. The grading of Floer or symplectic homology is
inessential for our purposes and we make no assumptions on $c_1(TW)$.
As usual, denote by $\WW$ the symplectic completion of $W$, i.e.,
$$
\WW=W\cup_M M\times [1,\,\infty)
$$
with the symplectic form $\omega=d\alpha$ extended from $W$ to
$M\times [1,\infty)$ as
$$
\omega := d(r\alpha),
$$
where $r$ is the coordinate on $[1,\,\infty)$. Sometimes it is
convenient to have the function $r$ also defined on a collar of
$M=\p W$ in $W$. Thus we can think of $\WW$ as the union of $W$ and
$M\times [1-\eta,\,\infty)$ for small $\eta>0$ with
$M\times [1-\eta,\, 1]$ lying in $W$ and the symplectic form given by
the same formula.

Unless specifically stated otherwise, most of the Hamiltonians
$H\colon \WW\to \R$ considered in this paper are constant on $W$ and
depend only on $r$ outside $W$, i.e., $H=h(r)$ on
$M\times [1,\,\infty)$, where the $C^\infty$-smooth function
$h\colon [1,\,\infty)\to \R$ is required to meet the following three
conditions:
\begin{itemize}
\item $h$ is strictly monotone increasing;
\item $h$ is convex, i.e., $h''\geq 0$, and $h''>0$ on $(1,\, \rmax)$
  for some $\rmax>1$ depending on~$h$;
\item $h(r)$ is linear, i.e., $h(r)=ar-c$, when $r\geq \rmax$.
\end{itemize}
In other words, the function $h$ changes from a constant on $W$ to
convex in $r$ on $M\times [1,\, \rmax]$, and strictly convex on the
interior, to linear in $r$ on $M\times [\rmax,\, \infty)$.

We will refer to $a$ as the \emph{slope} of $H$ (or $h$) and write
$a=\slope(H)$. The slope is often, but not always, assumed to be
outside the action spectrum of $\alpha$, i.e.,
$a\not\in\CS(\alpha)$. We call $H$ \emph{admissible} if
$H|_W=\const<0$ and \emph{semi-admissible} when $H|_W\equiv 0$. (This
terminology differs from the standard usage, and we emphasize that
\emph{admissible Hamiltonians are not semi-admissible}.) When $H$
satisfies only the last of the three conditions, we call it
\emph{linear at infinity}.

The difference between admissible and semi-admissible Hamiltonians is
just an additive constant: $H- H|_W$ is semi-admissible when $H$ is
admissible. Hence the two Hamiltonians have the same filtered Floer
homology up to an action shift. For our purposes, semi-admissible
Hamiltonians are notably more suitable due to the $H|_W\equiv 0$
normalization.

The Hamiltonian vector field $X_H$ is determined by the condition
$$
\omega(X_H,\, \cdot)=-dH,
$$
and, on $M\times [1,\,\infty)$, 
$$
X_H=h'(r) R_\alpha,
$$
where $R_\alpha$ is the Reeb vector field. We denote the Hamiltonian
flow of $H$ by $\varphi_H^t$, the Reeb flow of $\alpha$ by
$\varphi_\alpha^t$, where $t\in \R$, and the Hamiltonian
diffeomorphism generated by $H$ by $\varphi_H:=\varphi_H^1$.

Throughout the paper, by a \emph{$\tau$-periodic orbit} of $H$ we will
mean one of several closely related but distinct objects. It can be a
$\tau$-periodic orbit of $\varphi_H$ and then
$\tau\in\N$. Alternatively, it can stand for a $\tau$-periodic orbit
of the flow $\varphi_H^t$ with $\tau\in (0,\,\infty)$. Furthermore,
working with periodic orbits of flows or maps, we might or not have
the initial condition fixed. For instance, without an initial
condition fixed, a non-constant 1-periodic orbit of the flow of $H$
gives rise to a whole circle of 1-periodic orbits (aka fixed points)
of $\varphi_H$. Likewise, a prime $\tau$-periodic orbit of $\varphi_H$
comprises $\tau$ $\tau$-periodic points.  In most cases the exact
meaning should be clear from the context and is often immaterial; when
the difference is essential we will specify whether an orbit is of the
flow or the diffeomorphism and if the initial condition is fixed or
not.

Every $T$-periodic orbit $z$ of the Reeb flow with $T<a=\slope(H)$
gives rise to a 1-periodic orbit $\tz=(z,r_*)$ of the flow of $H$ with
$r_*$ determined by the condition
\begin{equation}
  \label{eq:level}
h'(r_*)=T.
\end{equation}
Clearly, $\tz$ lies in the shell $1<r<\rmax$, and we have a one-to-one
correspondence between 1-periodic orbits of $H$ and the periodic
orbits of $\varphi_\alpha^t$ with period $T<a$ whenever
$a\not\in \CS(\alpha)$. In the pair $\tz=(z,r_*)$, we usually view
$\tz$ as a 1-periodic orbit of the flow $\varphi^t_H$ of $H$ or a
circle of fixed points of the Hamiltonian diffeomorphism $\varphi_H$,
while $z$, contrary to what the notation might suggest, is
parametrized by the Reeb flow but not as a projection of $\tz$ to
$M$. (By \eqref{eq:level}, the two parametrizations of $z$ differ by
the factor of $h'(r_*)=T$.) Fixing an initial condition on $z$
determines an initial condition on $\tz$, and the other way around. In
particular, $z$ gives rise to a whole circle $\tz(S^1)$ of fixed
points of $\varphi_H$.

We say that a $T$-periodic orbit $z$ of the Reeb flow is
\emph{isolated} (as a periodic orbit) if for every $T'> T$ it is
isolated among periodic orbits with period less than $T'$. Clearly,
all periodic orbits of $\alpha$ are isolated if and only if for every
$T'$ the number of periodic orbits with period less than $T'$ is
finite.  For instance, a non-degenerate periodic orbit is isolated.
Note that $\tz$ is isolated as a 1-periodic orbit of the flow of $H$
if $z$ is isolated. No fixed point of $\varphi_H$ on $\tz(S^1)$ is
isolated, but $\tz$ is Morse--Bott non-degenerate, as the set of fixed
points $\tz(S^1)$, if and only if $z$ is non-degenerate;
cf.~\cite{Bo}.

The action functional $\CA_H$ is defined by
$$
\CA_H(\gamma)=\int_\gamma\hat{\alpha}-\int_{S^1} H(\gamma(t))\, dt,
$$
where $\gamma\colon S^1=\R/\Z\to \WW$ is a smooth loop in $\WW$ and
$\hat{\alpha}$ is the Liouville primitive $\alpha$ of $\omega$ on $W$
and $\hat{\alpha}=r\alpha$ on $M\times [1-\eta,\,\infty)$ for a
sufficiently small $\eta>0$. More explicitly, when
$\gamma\colon S^1\to M\times [1,\,\infty)$, we have
$$
\CA_H(\gamma)= \int_{S^1} r(\gamma(t))\alpha\big(\gamma'(t)\big)\, dt
- \int_{S^1} h\big(r(\gamma(t))\big)\, dt.
$$
Thus when $\gamma=\tz=(z,r_*)$ is a 1-periodic orbit of $H$, the
action can be expressed as a function of $r_*$ only:
$$
\CA_H(\tz)=A_h(r_*),
$$
where
\begin{equation}
  \label{eq:AH}
  A_h\colon [1,\,\infty)\to [0,\,\infty)\textrm{ is given by } A_h(r)=r
  h'(r)-h(r).
\end{equation}
Sometimes we will also denote this \emph{action function} by $A_H$.
This is a monotone increasing function, for
$$
A_h'(r)=h'(r)+ r h''(r)-h'(r)=rh''(r)\geq 0.
$$
It is not hard to show that
\begin{equation}
  \label{eq:maxAH}
\max A_h=A_h(\rmax)=c\geq a;
\end{equation}
see \cite[Sect.\ 2.1]{CGGM}. For this reason, we will in some
instances limit the domain of this function to $[1,\,\rmax]$.

While the function $A_h$ expresses the Hamiltonian action as a
function of $r$, we will also need another variant $\fa_H$ of the
action function, expressing the Hamiltonian action as a function of
the period $T$, i.e., the contact action. In other words, the function
$\fa_H$ translates the contact action to the Hamiltonian action. Thus
$$
\fa_H:=A_h\circ (h')^{-1}\colon [0,\,a]\to [0,\,\max A_h=A_h(\rmax)]
$$
is more specifically defined by the condition
\begin{equation}
  \label{eq:fa}
\fa_H(T)=A_h(r),\textrm{ where } h'(r)=T.
\end{equation}
Since $H$ is semi-admissible, $h'$ is one-to-one on $[1,\, \rmax]$,
and the inverse $(h')^{-1}$ is defined on $[0,\, a]$.  Then, using the
chain rule, we have
\begin{equation}
\label{eq:fa-der}
\fa'_H(T)=r:=(h')^{-1}(T)\textrm{ and } 1\leq \fa'_H\leq \rmax.
\end{equation}
Thus $\fa_H$ is a strictly monotone increasing, convex $C^1$-function,
which is $C^\infty$ on $(0,\,a )$, with $\fa''_h=\infty$ at $T=0$ and
$T=a$. Furthermore,
\begin{equation}
  \label{eq:fa-ineq}
\fa_{H_1}\leq \fa_{H_0} \textrm{ on } [0,\,\slope(H_0)] \textrm{
    whenever } H_1\geq H_0.
\end{equation}

A simple way to prove \eqref{eq:fa-ineq} is as follows, \cite[Sect.\
2.1]{CGGM}. First, note that that $-A_h(r)$ is the ordinate of the
intersection of the tangent line to the graph of $h$ at $(r,h(r))$
with the vertical axis. Furthermore, $\fa_{H_1}(T)=A_{h_1}(r_1)$ and
$\fa_{H_0}(T)=A_{h_0}(r_0)$, where $h'_1(r_1)=T=h'_0(r_0)$. Hence, the
two tangent lines have the same slope $T$. The tangent line to the
graph of $h_1$ lies above the tangent line to the graph of $h_0$; for
it passes through the point $(r_1, h_1(r))$ which is above the graph
of $h_0\leq h_1$. Therefore, $A_{h_0}(r_0)\geq A_{h_1}(r_1)$.

Furthermore, it is not
hard to see that
\begin{equation}
  \label{eq:fa-limit}
\fa_{sH}(T)\to T \textrm{ as } s\to\infty
\end{equation}
uniformly on compact sets in $[0,\,\infty)$ whenever $H$ is
semi-admissible.

\section{Filtered Floer and symplectic homology}
\label{sec:Floer}

In this section we recall basic definitions and results from Floer
theory used in the proof of Theorem \ref{thm:B}. Many, but not all, of
the constructions here are quite standard and go back to \cite{CFH,
  Vi} and can also be found in numerous subsequent accounts.

\subsection{Floer equation}
\label{sec:Floer-eq}
Fix an almost complex structure $J$ on $\WW$ satisfying the following
conditions:
\begin{itemize}
\item $J$ is compatible with $\omega$, i.e., $\omega(\cdot,\, J\cdot)$
  is a Riemannian metric,
\item $J r\p /\p r=R_\alpha$ on the cylinder $M\times [1,\,\infty)$,
\item $J$ preserves $\ker (\alpha) $.  
\end{itemize}
The last two conditions are equivalent to that
\begin{equation}
  \label{eq:complex_strc}
dr\circ J=-r\alpha.
\end{equation}
We call such almost complex structures \emph{admissible}. If the first
condition still holds on $\WW$, and the second and the third
conditions are met only outside a compact set while within a compact
set $J$ can be time-dependent and 1-periodic in time, we call $J$
\emph{admissible at infinity}.

Next, let $H$ be a Hamiltonian linear at infinity and let $J$ be an
admissible at infinity almost complex structure.  Following
\cite{CGGM, GG:LS}, it is convenient for our purposes to adopt the
$L^2$-anti-gradient of $\CA_H$,
\begin{equation}
  \label{eq:floer_1}
\p_s u =-\nabla_{L^2}\CA_H(u),
\end{equation}
as the Floer equation, where $u\colon \R\times S^1\to \WW$ and $(s,t)$
are the coordinates on $\R\times S^1$ with $S^1=\R/\Z$. Hence the
function $s\mapsto \CA_H\big(u(s,\,\cdot)\big)$ is decreasing.
Explicitly, this equation reads
\begin{equation}
\label{eq:floer_2}
\p_s u-J\big(\p_t u- X_H(u)\big)=0.
\end{equation}
Note that the leading term of this equation is the $\p$-operator, as
opposed to the $\bar{\p}$-operator as in the standard conventions.  In
other words, when $H\equiv 0$, solutions of \eqref{eq:floer_2} are
anti-holomorphic rather than holomorphic curves. Nonetheless the
standard properties of the solutions of the Floer equation readily
translate to our setting, e.g., via the change of variables
$s\mapsto -s$. We will often refer to solutions $u$ of the Floer
equation as \emph{Floer cylinders}. Recall that the energy of $u$ is
by definition
$$
E(u)=\int_{S^1\times\R}\|\p_s u\|^2\,dt\,ds.
$$

Let us assume from now on that $J$ is admissible and $H$ is
(semi-)admissible. Then the Floer equation is translation and rotation
invariant since, $J$ and $H$ are independent of $t$ (autonomous) and
$s$. Thus, whenever $u$ is a Floer cylinder,
$(s,t)\mapsto u(s+s_0,t+t_0)$ is also a Floer cylinder for all
$(s_0,t_0)\in \R\times S^1$. In particular, $u$ is never regular
unless $u$ is independent of $t$. Recall, however, that in the
notation from Section \ref{sec:setting}, $\tz=(z,r_*)$ is Morse--Bott
non-degenerate if and only if $z$ is non-degenerate; cf.~\cite{Bo}.

Let $u\colon \R \times S^1 \to\WW$ be a Floer cylinder for $H$. We say
that $u$ is \emph{asymptotic} to $\tz$ at $\infty$ if there exists a
sequence $s_i\to \infty$ such that $u(s_i,\cdot)\to \tz$ in the
$C^1$-sense, up to the choice of the initial condition on $\tz$ which
might depend on $s_i$.  (In other words, here we view $\tz$ as a
1-periodic orbit of the flow of $H$ without fixing an initial
condition and the choice of the initial condition turns it into a
1-periodic orbit of $\varphi_H$.)  This definition is equivalent to
that $u(s, \cdot)\to \tz$ in the $C^\infty$-sense as $s\to +\infty$
when $z$ is non-degenerate, and hence $\tz$ is Morse--Bott
non-degenerate. Moreover, in this case $u(0,s)$ converges as
$s\to\infty$, \cite{Bo}. (Likewise, $u$ is said to be asymptotic to
$\tz$ at $-\infty$ when $s_i\to -\infty$, etc.)

In general, $u$ can be asymptotic to more than one orbit $\tz$ at the
same end. However, $\CA_H(\tz)=\lim\CA_H\big(u(s_i,\cdot)\big)$, and
hence $\CA_H(\tz)$ is independent of the choice of $\tz$. Furthermore,
\eqref{eq:Energy-Action} below holds: $E(u)$ is the difference of
actions of the orbits which $u$ is asymptotic to at $\pm\infty$. It is
a standard fact that $u$ is asymptotic to some 1-periodic orbits of
$H$ at $\pm\infty$ if and only if $E(u)<\infty$; see \cite[Sec.\
1.5]{Sa}.

Next, assume that $u$ is asymptotic to $\tz$ at $\infty$ and $z$ is
isolated or, equivalently, $\tz$ is isolated as a 1-periodic orbit of
the flow of $H$. Then, as is easy to see, $\tz$ is unique (as a
1-periodic orbit of $\varphi_H^t$) and $u(s,\cdot)\to \tz$ as
$s\to\infty$ in the $C^1$-sense, up to the choice of an initial
condition on $\tz$ which might depend on $s$. This is a consequence of
the fact that
$$
E\big(u|_{[s_i, \infty)\times S^1}\big)\to 0\textrm{ as } s_i\to
\infty
$$
since $\CA_H\big(u(s,\cdot)\big)$ is a monotone function of $s$ and of
the argument in \cite[Sec.\ 1.5]{Sa}.

Let $u$ be asymptotic to $\tx=(x, r^+)$ at $-\infty$ and $\ty=(y,r^-)$
at $+\infty$. Then
\begin{equation}
  \label{eq:Energy-Action}
E(u)=\CA_H(\tx)-\CA_H(\ty)=A_H(r^+)-A_H(r^-).
\end{equation}
Here $r^+\geq r^-$ -- hence the notation -- since \eqref{eq:floer_1}
is an anti-gradient Floer equation and $A_H$ is an increasing
function.

We will make extensive use of two standard properties of Floer
cylinders $u$ for admissible or semi-admissible Hamiltonians $H$ and
admissible almost complex structures $J$.

The first one is the maximum principle asserting that the function
$r\circ u$ cannot attain a local maximum in the domain in
$\R\times S^1$ mapped by $u$ into $M\times [1,\,\infty)$, i.e., where
$r$ is defined. (We refer the reader to, e.g., \cite{Vi} and also
\cite[Sec.\ 2]{FS} for a direct proof of this fact.) The same is true
for $H$ linear and $J$ admissible at infinity, in the domain where
$H=ar-c$ and $J$ is admissible. Moreover, the maximum principle also
holds for continuation Floer trajectories when $h(r)=a(s)r-c(s)$ and
$a(s)$ is a non-decreasing function of $s$, with no constraints on the
function $c(s)$.  This version of the maximum principle is crucial to
having Floer cylinders and continuation solutions of the Floer
equation contained in a compact region of $\WW$, and hence the Floer
homology and continuation maps for homotopies with non-decreasing
slope are defined.

In particular, let, as above, $u$ be a solution of the Floer equation
asymptotic to 1-periodic orbits $\tx =(x,r^+)$ at $-\infty$ and
$\ty=(y,r^-)$ at $+\infty$. Then, by the maximum principle,
\begin{equation*}
\sup_{\R \times S^1} r\big(u(s,t) \big) \leq r^+
= r\big(u(-\infty,t)\big).
\end{equation*}

The second fact we will use is that $E(u)>\eps$ when $u$ enters
$\eta$-deep into $W$, i.e., $u$ is not entirely contained in
$M\times [1-\eta,\,\infty)$, for some $\eps>0$ depending on $J$ and
$\eta$ but independent of $u$ and a (semi-)admissible Hamiltonian
$H$. This is an immediate consequence of monotonicity since $H=\const$
in $W$, and hence $u$ is an (anti-)holomorphic curve; see, e.g.,
\cite{Si}.

\subsection{Floer homology and continuation maps}
\label{sec:Floer-cont}
Fix a ground field $\F$ which we will suppress in the notation. Let
$H$ be a Hamiltonian $H$ linear at infinity.  Assume first that
$\slope(H)\not\in\CS(\alpha)$. Then, regardless of whether $H$ is
non-degenerate or not, the filtered (contractible) Floer homology
$\HF^\tau(H)$ over $\F$ is readily defined as long as $\tau\in\R$ is
outside the action spectrum $\CS(H)$ of $H$. This is simply the
homology $\HF^\tau(\tH)$ of the Floer complex of a small
non-degenerate perturbation $\tH$ of $H$ with $\slope(\tH)=\slope(H)$,
generated by the 1-periodic orbits with action less than $\tau$. (Here
we treat $\HF^\tau(H)$ as an ungraded vector space over $\F$.) It is
easy to see that $\HF^\tau(\tH)$ is independent of $\tH$ when $\tH$ is
sufficiently close to $H$. The total Floer homology $\HF(H)$ is
$\HF^\infty(H)$ or, more precisely, $\HF^\tau(H)$ where
$\tau>\supp\CS(H)$.

Clearly, for $\tau_1\leq \tau_2$ we have the ``inclusion'' map
\begin{equation}
  \label{eq:incl}
\HF^{\tau_1}(H)\to \HF^{\tau_2}(H).
\end{equation}

As in Section \ref{sec:persistence}, we use \eqref{eq:semi-cont} to
extend this definition of $\HF^\tau(H)$ to all $\tau\in \R$. Namely,
for any $\tau\in\R$ which is now allowed to be in $\CS(H)$ or
$\tau=\infty$, we set
\begin{equation}
  \label{eq:left-limit}
\HF^\tau(H):=\varinjlim_{\tau'\leq \tau}\HF^{\tau'}(H),
\end{equation}
where we require that $\tau'\not\in\CS(H)$. The ``inclusion'' maps
naturally extend to these homology spaces and with this definition the
family of spaces $\tau\mapsto \HF^\tau(H)$ becomes a persistence
module. These maps are isomorphisms as long as the interval
$[\tau_1,\,\tau_2)$ contains no points of $\CS(H)$. (We changed the
notation for the persistence module parameter $s$ to $\tau$, for $s$
is taken by the homotopy parameter below.) In what follows we will be
interested in the number $\fbdyn_\eps(H)$ of bars of length greater
than $\eps>0$ in the barcode of this persistence module beginning
below $\fa_H(a)-\eps$, where $a=\slope(H)$; cf.\ Section
\ref{sec:persistence-bar}.

Note that 
\begin{equation}
  \label{eq:HF0}
\HF^\tau(H)=0 \textrm{ for }\tau\leq 0
\end{equation}
when $H$ is semi-admissible.

\begin{Remark}
  Alternatively, in a more \emph{ad hoc} fashion, one could have set
$$
\HF^\tau(H):=\varinjlim_{\tH\leq H}\HF^{\tau}(\tH),
$$
where $\tH$ is again non-degenerate, $\slope(\tH)=\slope(H)$, but now
$\tH\leq H$ pointwise and $\tau\not\in \CS(\tH)$. However, with our
conventions, this definition would \emph{not} be literally equivalent
to the one above and \eqref{eq:left-limit} would not hold in
general. In other words, $\tau\mapsto \HF^\tau(H)$ would not be a
persistence module in the sense of Section \ref{sec:persistence}:
Left-semicontinuity, \ref{PM3} and \eqref{eq:semi-cont}, would break
down. For instance, assume that $H\equiv \tau$ on $W$ and $H>\tau$ on
$M\times (0,\,\infty)$, e.g., $\tau=0$ and $H$ is
semi-admissible. Then we would have $\HF^\tau(H)=\H_*(W,M)\neq 0$ but
$\HF^{\tau'}(H)=0$ for all $\tau'<\tau$.
\end{Remark}

Let $H_s$, $s\in\R$, be a homotopy between two linear at infinity
Hamiltonians $H_0$ and $H_1$, i.e., $H_s$ is a family of linear at
infinity Hamiltonians such that $H_s=H_0$ when $s$ is close to
$-\infty$ and $H_s=H_1$ when $s$ is close to $+\infty$. (In what
follows we will take the liberty to have homotopies parametrized by
$s\in [0,\,1]$ or some other finite interval rather than $\R$.) There
are two cases where a homotopy gives rise to a map in Floer homology.

The first one is when all Hamiltonians $H_s$ have the same slope. Then
the homotopy induces a continuation map
$$
\HF^\tau(H_0)\to \HF^{\tau +C}(H_1)
$$
shifting the action filtration by
$$
C=\int_{-\infty}^\infty
\, \max_{z\in\WW} \, \max\{0,-\p_s H_s (z)\}\,ds.
$$
Moreover, it is well-known and not hard to show that $\HF(H)$ does not
change as long as $\slope(H)$ stays outside of $\CS(\alpha)$.

The second case is when $H_s$ is monotone increasing, i.e., the
function $s\mapsto H_s(z)$ is monotone increasing for all $z\in
\WW$. In particular, the function $s\mapsto \slope(H_s)$ is also
monotone increasing. Note that while $\slope(H_0)$ and $\slope(H_1)$
are still required to be outside $\CS(\alpha)$, the intermediate
slopes $\slope(H_s)$ can pass through the points of
$\CS(\alpha)$. Such a homotopy induces a map
$$
\HF^\tau(H_0)\to \HF^{\tau}(H_1)
$$
preserving the action filtration.

In both cases the fact that the continuation Floer trajectories are
confined to a compact set is a consequence of the maximum principle;
see Section \ref{sec:Floer-eq}.

The Floer homology is insensitive to perturbations of the Hamiltonian
$H$ and $\tau$ as long as $\slope(H)\not\in\CS(\alpha)$ and $\tau$
remains outside $\CS(H)$. To be more precise, fix a linear at infinity
Hamiltonian $H$ and $\tau$ meeting these conditions. Assume that the
slope of $H'$ is sufficiently close to the slope of $H$ and $H'$ is
$C^0$-close to $H$ on the complement of the domain where they both are
linear functions of $r$, and that $\tau'$ is close to $\tau$. Then
there is a natural isomorphism of the Floer homology groups
\begin{equation}
  \label{eq:invariance}
  \HF^\tau(H)\cong \HF^{\tau'}(H').
\end{equation}

Our next goal is to eliminate the assumption that
$a:=\slope(H)\not\in \CS(\alpha)$. The most important case in our
setting is that of the total Floer homology, i.e., $\tau>\sup \CS(H)$
for a semi-admissible Hamiltonian $H$, and we will focus on this
case. Then $V_s:=\HF(sH)$ is defined for all $s>0$ with
$sa\not\in\CS(\alpha)$. Moreover, the homotopy $H_s:=sH$ is monotone
increasing and these spaces are connected by continuation maps. (This
would not be true if $H$ were admissible rather than semi-admissible.)
In what follows, it is essential to extend the definition of $V_s$ to
all $s\in\R$ and turn $V_s$ into a persistence module. To this end, we
set $V_s=0$ for $s\leq 0$. When $s>0$ and $s a\in\CS(\alpha)$ we use
\eqref{eq:semi-cont} as in Section \ref{sec:persistence}:
\begin{equation}
  \label{eq:Fl-pers}
V_s:=\varinjlim_{s'<s}\HF(s' H),
\end{equation}
where $s'a\not\in\CS(\alpha)$. Clearly, $\{V_s\}$ is indeed a
persistence module. We will denote the number of bars of length
greater than $\eps>0$ beginning below $s$ in the barcode of $V$ by
$\fbfl_\eps(s, H)$.

\begin{Remark}
  Alternatively and more generally, for any Hamiltonian $H$ linear at
  infinity, we could have defined the filtered Floer homology as
$$
\HF^\tau(H):=\varinjlim_{H'\leq H}\HF^\tau(H'),
$$
where the limit is taken over Hamiltonians $H'\leq H$ linear at
infinity with $\slope(H')\not \in \CS(\alpha)$. One can show that,
when $H$ is (semi-)admissible, we may require $H'$ to be
(semi-)admissible and that this definition agrees with
\eqref{eq:Fl-pers} in the sense that
$$
\HF(H):=\varinjlim_{0<s<1}\HF(sH)
$$
when $H$ is semi-admissible; cf.\ Remark \ref{rmk:adm}.  However, for
our purposes, the definition, \eqref{eq:Fl-pers}, is more convenient
as it fits better in the general framework of persistence modules; see
Section \ref{sec:persistence}.
\end{Remark}

\subsection{Symplectic homology}
\label{sec:SH}
In this section we review the definition and properties of filtered
symplectic homology, focusing on its relations to the filtered Floer
homology of semi-admissible Hamiltonians. These relations are somewhat
less standard than the material from the previous two sections. Our
treatment of the question has some overlaps with, e.g., \cite{AM, Me},
although the setting and emphasis there are different, and also more
directly with \cite{CGGM}.

The \emph{symplectic homology} $\SH^\tau(\alpha)$, where
$\tau>0$, is defined as
\begin{equation}
  \label{eq:SH}
\SH^\tau(\alpha):=\varinjlim_H \HF^\tau(H),
\end{equation}
where traditionally the limit is taken over all Hamiltonians linear at
infinity and such that $H|_W<0$. Since admissible (but not
semi-admissible) Hamiltonians form a co-final family, we can limit $H$
to this class. Furthermore, we set
\begin{equation}
  \label{eq:SH0}
  \SH^\tau(\alpha):=0\textrm{ when } \tau\leq 0.
\end{equation}
When working with this definition, it is useful to keep in mind that,
by \eqref{eq:fa-limit},
\begin{equation}
  \label{eq:spectra-conv}
\CS(H)\to\{ 0\}\cup\CS(\alpha)
\end{equation}
uniformly on compact intervals. 

\begin{Remark}[Cofinal sequences]
  \label{rmk:adm}
  In \eqref{eq:SH}, with \eqref{eq:SH0} in mind, we could have
  required that $H|_W\leq 0$ rather than that $H|_W<0$, or
  equivalently allowed $H$ to be semi-admissible or admissible. This
  would result in the same groups $\SH^\tau(\alpha)$. Indeed, let $H$
  be a semi-admissible Hamiltonian. Pick two sequences of positive
  numbers: $s_i\to\infty$ and $\eps_i\to 0$. Then the sequence
  $H_i=s_i H-\eps_i$ is co-final in the class of admissible
  Hamiltonians.
\end{Remark}

The definition of symplectic homology via a direct limit,
\eqref{eq:SH}, over admissible or even semi-admissible Hamiltonians is
quite inconvenient for our purposes. In fact, the limit over a much
smaller class of Hamiltonians is sufficient:

\begin{Lemma}
  \label{lemma:sympl-Fl}
  Let $H$ be any semi-admissible Hamiltonian. Then we have
\begin{equation}
  \label{eq:sympl-Fl}
\SH^\tau(\alpha)=\varinjlim_{s\to\infty} \HF^\tau(sH)
\end{equation}
for any $\tau\leq \infty$.
\end{Lemma}
Clearly, similar statement holds for any action interval. We will
prove the lemma a bit later in this section. In fact, passing to a
limit in the definition of the symplectic homology is not needed at
all if one is willing to make concessions of restricting the action
range from above and also slightly reparametrizing the action.

\begin{Theorem}
  \label{thm:sympl-Fl}
  Let $H$ be a semi-admissible Hamiltonian with $\slope(H)=a$. Then,
  for every $\tau\leq a$, there exists an isomorphism
$$
\Phi_H^\tau\colon \SH^\tau(\alpha) \stackrel{\cong}{\longrightarrow}
\HF^{\fa_H(\tau)}(H),
$$
where the function $\fa_H$ turning the Reeb period (aka the contact
action) into the Hamiltonian action is defined by \eqref{eq:fa} in
Section \ref{sec:setting}.

Moreover, these isomorphisms are natural in the sense that they
commute with the ``inclusion'' and monotone continuation maps. To be
more precise, for any $\tau'\leq \tau\leq a$ and two semi-admissible
Hamiltonians $H'\leq H$, the diagrams
$$
\begin{tikzcd}[row sep=large]
&\SH^{\tau'}(\alpha)
\arrow[r,"\Phi_H^{\tau'}"]
\arrow[d]
&
\HF^{\fa_H(\tau')}(H)
\arrow[d]
\\
&
\SH^{\tau}(\alpha)
\arrow[r,"\Phi_H^{\tau}"]
&
\HF^{\fa_H(\tau)}(H)
\end{tikzcd}
$$
and
$$
\begin{tikzcd}[row sep=large]
&\SH^{\tau}(\alpha)
\arrow[r,"\Phi_{H'}^{\tau}"]
\arrow[d, "\id"]
&
\HF^{\fa_{H'}(\tau)}(H')
\arrow[d]
\\
&
\SH^{\tau}(\alpha)
\arrow[r,"\Phi_{H}^{\tau}"]
&
\HF^{\fa_{H}(\tau)}(H)
\end{tikzcd}
$$
commute, where the vertical arrows are the ``inclusion'' maps in the
first diagram and the right vertical arrow is the monotone
continuation map in the second.
\end{Theorem}

The first consequence of the theorem, central to this paper, is the
fact that the filtered symplectic homology defined as above is a
persistence module. (This does not directly follow the definition.)

\begin{Corollary}
 \label{cor:SH-persistence} 
 The family of vector spaces $\tau\mapsto \SH^\tau(\alpha)$ is a
 persistence module in the sense of Section \ref{sec:persistence} with
 structure maps defined as the direct limit of the ``inclusion'' maps
 \eqref{eq:incl}.
\end{Corollary}

\begin{proof}
  Conditions \ref{PM1}, \ref{PM2} and \ref{PM4} from Section
  \ref{sec:persistence} are obviously satisfied with $s_0=0$ and
  $\CS:=\CS(\alpha)\cup\{0\}$. Left-semicontinuity, \ref{PM3}, follows
  from the first commutative diagram together with the facts that the
  function $\fa_H$ is continuous and $t\mapsto \HF^t(H)$ is a
  persistence module.
\end{proof}

The second consequence of the theorem is that to obtain the filtered
symplectic homology for a finite range of action it suffices to take a
semi-admissible Hamiltonian with an appropriate slope without passing
to the limit; cf.\ \cite{Vi}. Indeed, setting $\tau=\infty$ or just
$\tau>\sup\CS(H)$ we have the following.

\begin{Corollary}
 \label{cor:sympl-Fl} 
 For any semi-admissible Hamiltonian $H$ with 
 $\slope(H)=a$,
 $$
\SH^a(\alpha)\cong \HF(H).
$$
Moreover, whenever $H'\leq H$ are semi-admissible with $a'=\slope(H')$
and $a=\slope(H)$, the diagram
$$
\begin{tikzcd}[row sep=large]
&\SH^{a'}(\alpha)
\arrow[r,"\cong"]
\arrow[d]
&
\HF(H')
\arrow[d]
\\
&
\SH^{a}(\alpha)
\arrow[r,"\cong"]
&
\HF(H)
\end{tikzcd}
$$
commutes, where the left vertical arrow is the structure or
``inclusion'' map and the right vertical arrow is the continuation
map.
\end{Corollary}

With the corollaries stated, we conclude this section by proving Lemma
\ref{lemma:sympl-Fl} and Theorem \ref{thm:sympl-Fl}.

\begin{proof}[Proof of Lemma \ref{lemma:sympl-Fl}]
  The lemma is essentially a consequence of Remark \ref{rmk:adm} and
  the definitions.  When $\tau\leq 0$, the statement follows
  immediately from \eqref{eq:HF0} and \eqref{eq:SH0}. Ditto for
  $\tau=\infty$. Hence, we will assume that $0<\tau<\infty$ throughout
  the rest of the proof.

  Clearly, in \eqref{eq:sympl-Fl} we can replace the direct limit as
  $s\to \infty$ by the direct limit over any monotone increasing
  sequence $s_i\to\infty$, and the limit is independent of this
  sequence. Fix such a sequence and any monotone decreasing sequence
  $\eps_i\to 0^+$. Then $H_i:=s_i H-\eps_i$ is a cofinal sequence,
  which we can use in \eqref{eq:SH}. On the other hand,
$$
\HF^\tau(H_i)=\HF^{\tau-\eps_i}(s_iH).
$$

Assume first that $\tau\not\in\CS(\alpha)$.  Then
$\tau-\eps_i\not \in \CS(\alpha)$ for all large $i$ and, by
\eqref{eq:spectra-conv}, $\tau-\eps_i$ and $\tau$ are in the same
connected component of the complement to
$\CS(s_iH)\to \CS(\alpha)\cup\{0\}$. (This is the point where it is
essential that $\tau\neq 0$.) Therefore,
$$
\HF^{\tau-\eps_i}(s_iH)=\HF^{\tau}(s_iH)
$$
due to \eqref{eq:invariance}, and the statement again follows by
passing to the limit as $i\to\infty$. It is essential for the next
step that here the sequences $s_i\to\infty$ and $\eps_i\to 0^+$ are
arbitrary.

Now, assume that possibly $\tau\in \CS(\alpha)$. Pick a sequence
$s_i\to\infty$ so that $\tau\not\in \CS(s_iH)$. Then
$$
\HF^\tau(s_iH)=\HF^t(s_iH)
$$
for all $t\in [\tau-\delta_i, \, \tau+\delta_i]$ and for some
$\delta_i>0$ depending on $i$. Next, choose a sequence $\eps_i\to 0^+$
such that $\eps_i<\delta_i$. Then, by \eqref{eq:invariance},
$$
\HF^\tau(s_iH)=\HF^{\tau-\eps_i}(s_iH)=\HF^{\tau}(H_i)
$$
for all large $i$. Passing to the limit as $i\to\infty$, we obtain
\eqref{eq:sympl-Fl}.
\end{proof}

\begin{proof}[Proof of Theorem \ref{thm:sympl-Fl}]
  Let $H_0\leq H_1$ be two semi-admissible Hamiltonians. For the sake
  of simplicity we will assume that they have the same $\rmax$. (This
  assumption is not essential.) Consider the function
$$
f=f_{H_0,H_1}:=\fa_{H_1}\circ \fa_{H_0}^{-1}
\colon [0,\, A_{H_0}(\rmax)]\to [0,\, A_{H_1}(\rmax)].
$$
The function $f$ is monotone as a composition of two monotone
increasing functions and gives rise to a one-to-one correspondence
between the action spectra as long as the target is in the range of
$f$. Furthermore, $f(\tau)\leq \tau$ for all $\tau$; see
\cite{CGGM}. The proof of the theorem hinges on the following result.

\begin{Lemma}[Prop.\ 3.1, \cite{CGGM}]
  \label{prop:f}
  For all $\tau< A_{H_0}(\rmax)$, there are isomorphisms of the Floer
  homology groups
  \begin{equation}
    \label{eq:isom-f}
    \HF^\tau(H_0)\stackrel{\cong}{\longrightarrow} \HF^{f(\tau)}(H_1).
\end{equation}
These isomorphisms are natural in the sense that they commute with
``inclusion'' maps and monotone homotopies.
\end{Lemma}

Set $H_0=H$ and $H_1=s H$ for $s\geq 1$ and $f_s:=f_{H,sH}$. Then, for
any $\tau\leq a$ we have isomorphisms of the Floer homology groups
$$
\HF^\tau(H)\stackrel{\cong}{\longrightarrow} \HF^{f_s(\tau)}(sH).
$$
It is not hard to see from \eqref{eq:fa-limit} that
$f_s(\tau)\to \fa^{-1}_H(\tau)$ as $s\to \infty$. Passing to the limit
as $s\to\infty$ and applying Lemma \ref{lemma:sympl-Fl}, we obtain the
inverse of the desired isomorphism
$\Phi_H^{\fa^{-1}_H(\tau)}$. Naturality of these isomorphisms readily
follows from that the isomorphisms \eqref{eq:isom-f} are natural.
\end{proof}

\section{Barcode entropy revisited}
\label{sec:def-rev}
While Definition \ref{def:barcode_entropy} is simple and intuitive, it
is not very convenient to work with; for it is not directly connected
to the dynamics of the Reeb flow or the Hamiltonian flow of a
(semi-)admissible Hamiltonian. Nor is it obviously related to the
Floer homology of such Hamiltonians. In this section we rephrase the
definition of barcode entropy in several different ways to remedy this
shortcoming.  While completely self-contained, the discussion below
has substantial overlaps with \cite{FLS}, although our treatment of
the question is more brief. We start with some simple algebraic
observations.

\subsection{Barcode entropy of a persistence module}
\label{sec:persistence-bar}
The definition of barcode entropy of Reeb flows extends to general
persistence modules in a straightforward way. Namely, let $V=\{V_s\}$
be a persistence module. Denote by $\fb_\eps(s,V)$ or just
$\fb_\eps(s)$ the number of bars of length greater than $\eps>0$ and
beginning below $s$.  The \emph{barcode entropy} of $V$ is then
defined as
\begin{equation}
  \label{eq:hbar-gen}
\hbar_\eps(V)=\limsup_{s\to\infty}\frac{\log^+
  \fb_\eps(s,V)}{s}
\quad
\textrm{ and }
\quad
\hbar(V)=\lim_{\eps\searrow 0}\hbar_\eps(V).
\end{equation}

We say that a persistence module $W=\{W_s\}$ is a
\emph{reparametrization} of $V$ if $W_s=V_{\xi(s)}$, where the
function $\xi\colon \R\to \R$ is continuous, strictly monotone
increasing and onto. (The structure maps in the persistence module $W$
come from the structure maps in $V$.) It is not hard to see that
\begin{equation}
  \label{eq:reparametrization}
\fb_{c\eps}(s,W)\leq\fb_{\eps}\big(\xi(s),V\big)
\end{equation}
when $\xi^{-1}$ is Lipschitz with (global) Lipschitz constant $c$ or,
equivalently,
$$
\fb_{c'\eps}\big(\xi(s),V\big)\leq \fb_{\eps}(s,W),
$$
where $c'$ is the (global) Lipschitz constant of $\xi$.
As a consequence,
\begin{equation}
  \label{eq:hbar-reparametrization}
  \hbar(W)=a \hbar(V)
\end{equation}
whenever $\xi$ is bi-Lipschitz, i.e., both $\xi$ and $\xi^{-1}$ are
Lipschitz on $\R$, and $a=\lim \xi(s)/s$ as $s\to\infty$ assuming that
the limit exists. For instance, \eqref{eq:hbar-reparametrization}
holds when $\xi(s)=a s$ for some $a>0$.

Next, given a persistence module $V$, let us define a family of
persistence modules $V(s)$, $s\in \R$, by truncating $V$ at $s$, i.e.,
$V(s)_\tau=V_\tau$ when $\tau\leq s$ and $V(s)_\tau=V_s$ when
$\tau\geq s$. (Warning: this is not the standard notion of truncation;
cf.\ Example \ref{ex:Floer}.) In other words, the finite bars in $V$
ending below $s$ give rise to the finite bars of $V(s)$, the infinite
bars or finite bars of $V$ containing $s$ become the infinite bars of
$V(s)$, and the other bars disappear. In particular, all bars in
$V(s)$ begin in the interval $[0,\, s)$.

Then, by analogy with the barcode entropy of a Hamiltonian
diffeomorphism (see \cite{CGG:Entropy}), we associate to $V$ the
\emph{dynamics barcode entropy} as follows. Set $\fbdyn_\eps(s,V)$ to
be the number of bars of length greater than $\eps>0$ beginning below
$s-\eps$ in the barcode of $V(s)$. Thus
$\fbdyn_\eps(s,V)=\fb_\eps(s-\eps, V)$ is almost to the total number
of bars longer than $\eps$ in $V(s)$, up to an error coming from the
bars beginning in $[s-\eps,\, s)$. Then the dynamics barcode entropy
of $V$ is defined by
\begin{equation}
  \label{eq:hbar-V-dyn}
\hbardyn_\eps(V)=\limsup_{s\to\infty}\frac{\log^+
  \fbdyn_\eps(s,V)}{s }
\quad
\textrm{ and }
\quad
\hbardyn(V)=\lim_{\eps\searrow 0}\hbardyn_\eps(V).
\end{equation}

\begin{Remark}
  The reason that the bars are required to begin below $s-\eps$ is
  that the bars of $V$ containing $s$, no matter how short, give rise
  to infinite bars in $V(s)$. The results from \cite{As, Kal} indicate
  that in the Hamiltonian case the number of such bars can grow
  arbitrarily fast for some sequence $s_k\to\infty$ without any clear
  relation to the topological entropy of the underlying system. The
  requirement on the beginning of the bars keeps such bars from
  affecting the count.
\end{Remark}  

\begin{Example}
  \label{ex:Floer}
  In the setting we are interested in, $V$ is the persistence module
  $\SH^{s}(\alpha)$, and, by Corollary \ref{cor:sympl-Fl}, the
  truncated module $V(sa)$ is isomorphic to a reparametrization of the
  Floer homology persistence module $\tau\mapsto \HF^\tau(sH)$ when
  $H$ is semi-admissible with $a=\slope H$. This example motivates the
  choice of our rather uncommon truncation procedure. An alternative
  would be a more standard variant of truncation where $V(s)_\tau:=0$
  for $\tau>s$. This variant would also be suitable for our purposes
  and adopting it we would count the bars longer than $\eps$ beginning
  below $s$ (rather than $s-\eps$), although the version we use is
  more intuitive from the Floer theoretic perspective.
\end{Example}

Note that the dynamics barcode entropy can be defined for any family
of persistence modules. For the family of truncated persistence
modules $V(s)$, the dynamics barcode entropy coincides with the
ordinary barcode entropy as the following formal and nearly obvious
proposition shows.

\begin{Proposition}
  \label{prop:hbar-dyn-V}
  For any persistence module $V$ and any $\eps>0$, we have
\begin{equation}
    \label{eq:hbar-dyn-V}
    \hbar_\eps(V)=\hbardyn_\eps(V) \quad
    \text{ and hence }
    \quad
    \hbar(V)=\hbardyn(V).
\end{equation}
Furthermore, in \eqref{eq:hbar-gen} and \eqref{eq:hbar-V-dyn},
replacing the upper limits as $s\to\infty$ by the upper limit over any
monotone increasing sequence $s_k\to\infty$ such that
$s_{k+1}/s_k\to 1$ does not affect the definitions and
\eqref{eq:hbar-dyn-V} holds already on the level of $\eps$-entropy.
\end{Proposition}  

Hence, in what follows, we need not distinguish between these two
types of barcode entropy and will use the notation $\hbar$ for both of
them.

\begin{proof}
  By definition, for any $s\geq \eps>0$, we have
  $$
  \fb_\eps(s-\eps)=\fbdyn_\eps(s)\leq \fb_\eps(s),
  $$
  where we suppressed $V$ in the notation, and \eqref{eq:hbar-dyn-V}
  follows. To prove the moreover part, we focus on
  \eqref{eq:hbar-gen}. Clearly,
  $$
  \limsup_{k\to\infty}\frac{\log^+
    \fb_\eps(s_k)}{s_k}\leq \hbar_\eps
  $$
  for any sequence $s_k\to\infty$ and we only need to prove the
  opposite inequality.

  Let the sequence $s_k$ be as in the proposition. 
  Let $t_i\to\infty$ be such that
  $$
  \lim_{i\to\infty}\frac{\log^+
    \fb_\eps(t_i)}{t_i}= \hbar_\eps.
  $$
  For every $i$, pick $k_i$ so that $s_{k_i}\leq t_i<
  s_{k_i+1}$. Then
  $$
  \frac{s_{k_i+1}}{s_{k_i}}\frac{\log^+
    \fb_\eps(s_{k_i+1})}{s_{k_i+1}}=
  \frac{\log^+
    \fb_\eps(s_{k_i+1})}{s_{k_i}}
  \geq \frac{\log^+ \fb_\eps(t_i)}{t_i}.
  $$
  Passing to the (upper) limit and using the fact that
  $s_{k+1}/s_k\to 1$, we see that
  $$
  \limsup_{k\to\infty}\frac{\log^+ \fb_\eps(s_k)}{s_k}\geq
  \hbar_\eps,
  $$
  which concludes the proof of the proposition.
\end{proof}

\subsection{Barcode entropy via Floer homology}
\label{sec:barcode-Floer}
Let us now apply the observations from the previous section to Floer
homology.  Let $H$ be a semi-admissible Hamiltonian with slope $a$. As
in Section \ref{sec:Floer-cont}, consider the persistence module
$s\mapsto V_s$ formed by the Floer homology spaces $V_s:=\HF(s H)$
along with the continuation maps $\HF(s' H)\to \HF(s H)$ when
$s'\leq s$. By Corollary \ref{cor:sympl-Fl}, $V$ is isomorphic to the
symplectic homology persistence module $\SH^{as}(\alpha)$. Hence, the
number $\fbfl_\eps(s, H)$ of bars of length greater than $\eps>0$
beginning below $s$ in the barcode of $V$ is equal to
$\fb_\eps(as)$. As a consequence,
\begin{equation}
  \label{eq:hbar-fl}
\hbar_\eps(\alpha)=\frac{1}{a}\limsup_{s\to\infty}\frac{\log^+
  \fbfl_\eps(s, H)}{s}.
\end{equation}
From now on we will not distinguish the persistence modules
$\SH^{as}(\alpha)$ and $\HF(sH)$.

Alternatively, as has been pointed out in Section
\ref{sec:Floer-cont}, for a fixed Hamiltonian $H$ with slope $a$, we
can view the filtered Floer homology $\tau\mapsto\HF^\tau(H)$ as a
persistence module. Note that all bars in this persistence module
begin below $\fa_H(a)$. Let as in Section \ref{sec:persistence-bar}
$\fbdyn_\eps(H)$ be the number of bars of length greater than $\eps>0$
in its barcode beginning below $\fa_H(a)-\eps$. Then, analogously to
the definition of barcode entropy for Hamiltonian diffeomorphisms in
\cite{CGG:Entropy}, we set the \emph{dynamics barcode entropy} of $H$
to be
\begin{equation}
  \label{eq:hbar-dyn}
\hbardyn_\eps(H)=\limsup_{s\to\infty}\frac{\log^+
  \fbdyn_\eps(sH)}{s },
\end{equation}
and
$$
\hbardyn(H)=\lim_{\eps\searrow 0}\hbardyn_\eps(H).
$$
As readily follows from the definition,
$$
\hbardyn_\eps(c H)=c \cdot\hbardyn_\eps(H) \quad \text{ and } \quad
\hbardyn(cH)=c\cdot\hbardyn(H)
$$
for every $c>0$.

As in Section \ref{sec:persistence-bar}, denote by $V(s)$ the
persistence module obtained from the persistence module
$s\mapsto V_s:=\SH^{as}(\alpha)=\HF(sH)$ by truncating at $s$. It is
clear that $V(s)_\infty=\HF(sH)$. However, as persistence modules,
$V(s)$ and $\HF(sH)$ are in general different:
$V(s)_\tau\neq \HF^\tau(sH)$. Yet the two families of persistence
modules have the same barcode entropy as the following observation
along the lines of Proposition \ref{prop:hbar-dyn-V} asserts.

\begin{Theorem}
  \label{thm:two-defs}
  For any semi-admissible Hamiltonian $H$ with slope
  $a$, we have
  \begin{equation}
    \label{eq:two-defs}
  \hbar(\alpha)=\frac{\hbardyn(H)}{a}.
\end{equation}
Furthermore, in \eqref{eq:hbar-dyn}, replacing the upper limits as
$s\to\infty$ by the upper limit over any monotone increasing sequence
$s_k\to\infty$ such that $s_{k+1}/s_k\to 1$ does not affect the
definition and hence \eqref{eq:two-defs}.
\end{Theorem}

Due to this theorem we need not distinguish the barcode entropies
$\hbar$ and $\hbardyn$.  Note however that in contrast with
Proposition \ref{prop:hbar-dyn-V} or \eqref{eq:hbar-fl} we do not
claim here the equality on the level of $\eps$-barcode entropy, and we
believe that $\hbar_\eps(\alpha)\neq \hbardyn_\eps(H)/a$ in general.

\begin{Remark}[On the proof of Theorem \ref{thm:A}]
  \label{rmk:ThmA}
  Assume that $ka\not\in\CS(\alpha)$ for all $k\in\N$. Then taking
  $s_k=k$ in \eqref{eq:hbar-dyn} we arrive at the definition of
  barcode entropy for a semi-admissible Hamiltonian $H$ which is
  completely analogous to the definition of the barcode entropy for
  compactly supported Hamiltonian diffeomorphisms since
  $\varphi_{kH}=\varphi^k_H$.  Furthermore, the proof of \cite[Thm.\
  A]{CGG:Entropy} carries over word-for-word to semi-admissible
  Hamiltonians and, as a consequence, we arrive at Theorem \ref{thm:A}
  in a way somewhat different from \cite{FLS}. To be more specific, in
  this variant of the proof the tomograph construction is to be
  applied to the shell $U=\{1\leq r\leq\rmax\}$ resulting in the
  inequality
  $$
  \hbardyn(H)\leq \htop\big(\varphi_H|_U\big)=a\htop(\alpha).
  $$
\end{Remark}

\begin{proof}[Proof of Theorem \ref{thm:two-defs}]
  By Theorem \ref{thm:sympl-Fl}, for any semi-admissible Hamiltonian
  $H$ with slope $a$ there exists a natural isomorphism of vector
  spaces
$$
\HF^{\fa_H(\tau)}(H)\to \SH^\tau(\alpha)
$$
as long as $\tau \leq a$, and the function $\fa$ is bi-Lipschitz with
$1\leq \fa'_H\leq \rmax$. Hence, we also have isomorphisms for the
family of persistence modules
$$
\HF^{\fa_{sH}(\tau)}(sH)\to V(s)_\tau,
$$
where as above $V(s)$ is the persistence module
$V_s:=\SH^{as}(\alpha)=\HF(sH)$ truncated at $s$.  The
reparametrizations $\fa_{sH}$ are bi-Lipschitz uniformly in $s$:
$1\leq \fa'_{sH}\leq \rmax$ for all $s>0$ by \eqref{eq:fa-der}.

Therefore, similarly to \eqref{eq:reparametrization}, we have
$$
\fbdyn_{\eps}(sH)\geq \fb_\eps\big(V(s)\big)\geq
\fbdyn_{\rmax \eps}(sH).
$$
Passing to the upper limit as $s\to\infty$ and then to the limit as
$\eps\to 0$, we see that $\hbardyn(H)=\hbardyn(V)$. Now
\eqref{eq:two-defs} follows from Proposition \ref{prop:hbar-dyn-V} and
\eqref{eq:hbar-reparametrization}. The proof of the ``Furthermore''
part is nearly identical to its counterpart in the proof of
Proposition \ref{prop:hbar-dyn-V} and we omit it.
\end{proof}

\section{Crossing energy}
\label{sec:cross-energy+pf}
The key new ingredient of the proofs is the Crossing Energy Theorem
(Theorem \ref{thm:CE}). In this section we state the Crossing Energy
Theorem, which is proved in Section \ref{sec:cross}, and use it to
establish Theorem \ref{thm:B}.

\subsection{Crossing Energy Theorem}
\label{sec:CE}
Recall that a compact invariant set $K$ of the Reeb flow is said to be
\emph{locally maximal} or isolated if there exists a neighborhood
$U\supset K$ such that every invariant set contained in $U$ must be a
subset of $K$ or, equivalently, for every $x\in U\setminus K$ the
integral curve through $x$ is not entirely contained in $U$. We call
$U$ an \emph{isolating neighborhood}. For instance, a hyperbolic
periodic orbit is locally maximal. In general, a periodic orbit can be
isolated as a periodic orbit but not as an invariant set.  We refer
the reader to \cite[Sect.\ 17.4]{KH} or \cite{FH} for the definition
of a hyperbolic invariant set.

\begin{Theorem}[Crossing Energy Theorem, I]
  \label{thm:CE}
  Let $K$ be a hyperbolic, locally maximal compact invariant set of
  the Reeb flow of $\alpha$. Fix an interval
  $$
  I=[r_-,\, r_+]\subset (1,\,\rmax)
  $$
  and let $H(r,x)=h(r)$ be a semi-admissible Hamiltonian with
  $\slope(H)=:a \not\in\CS(\alpha)$ such that
\begin{equation}
  \label{eq:h'''1}
  h'''\geq 0 \textrm{ on } [1,\, r_++\delta]
\end{equation}
for some $\delta>0$ with $r_++\delta<\rmax$. Fix an admissible almost
complex structure $J$. Furthermore, let $z$ be a $T$-periodic orbit of
$\alpha$ in $K$ and $\tz:=(z, r^*)$ be the corresponding 1-periodic
orbit of the flow of $sH$. (Hence, $sa\geq T$ and $r^*$ depends on
$s$.) Assume that $r^*\in I$.

Then there exists $\sigma >0$ independent of $s$ and $z$ such that
$E(u) \geq \sigma$ for any Floer cylinder
$u \colon \R \times S^1 \to \WW$ of $sH$ asymptotic to $\tz$ at either
end, unless $E(u)=0$ and thus $u$ is trivial, i.e., independent of the
first coordinate.

\end{Theorem}

This theorem is a partial generalization of \cite[Thm.\ 4.1]{CGGM}
where $K$ is just one locally maximal periodic orbit of the Reeb flow.
A remark is due regarding the statement of Theorem \ref{thm:CE} which
we will prove in Section \ref{sec:cross}.

\begin{Remark}
  \label{rmk:CE}
  The point that $\sigma$ is independent of $s$ is crucial for our
  purposes. Moreover, without it, for a fixed $s$ the theorem would
  follow immediately from a suitable variant of Gromov compactness
  theorem under no assumptions on $K$ other than that for every $T$
  all $T$-periodic orbits in $K$ are isolated.

  Secondly, fix $0<\sigma'<\sigma$ and $s\geq 0 $. Consider a
  $C^\infty$-small $s$-periodic in time, non-degenerate perturbation
  $\tH$ of $sH$ and a $C^\infty$-small compactly supported generic
  $s$-periodic perturbation $\tJ$ of $J$. The 1-periodic orbit $\tz$
  of $sH$ from Theorem \ref{thm:CE} splits into several non-degenerate
  periodic orbits of $\tH$ contained in a small tubular neighborhood
  of $\tz$. It follows again from a suitable version of the Gromov
  compactness theorem (see, e.g., \cite{Fi}) that every Floer cylinder
  of $\tH$ asymptotic to any of these orbits at either end has energy
  greater than $\sigma'$.
\end{Remark}

\subsection{Proof of Theorem \ref{thm:B}}
\label{sec:pf-thmB}
The proof comprises four steps. For the sake of simplicity, we will
assume that $\charr\F=2$. If $\charr\F\neq 2$, we would need to throw
away 
bad orbits (i.e., even iterates of the orbits with odd number of
eigenvalues in $(-1,0)$) from the periodic orbit count in Steps 3 and
4; see, e.g., \cite[Sect.\ 3]{CGG:ReebHZ}. This, however, does not
affect the exponential growth rate of the orbits, and ultimately the
result holds for every ground field.

\medskip\noindent\emph{Step 1: The Hamiltonian.} Throughout the proof,
we treat the barcode entropy of $\alpha$ in the sense of Theorem
\ref{thm:two-defs}. Thus, let $H(x,r)=h(r)$ be a semi-admissible
Hamiltonian.  Without loss of generality, we may assume that
$a:=\slope(H)=1$.

We will require in addition that \eqref{eq:h'''1} is satisfied and
$a=1\not\in\CS(\alpha)$ so that we can apply Theorem \ref{thm:CE} to
$H$. Furthermore, we can make $h'(r_+)<1$ arbitrarily close to $1$. To
be more precise, it is not hard to show that for any $\eta>0$, there
exists a semi-admissible Hamiltonian $H$ with $a=1$ such that
\eqref{eq:h'''1} holds and
$$
1-\eta\leq h'(r_+).
$$

Recall that $\fbdyn_\eps(sH)$ is the number of bars of length greater
than $\eps>0$ beginning below $s-\eps$ in the barcode of
$V_\tau:=\HF^\tau(sH)$; see Section \ref{sec:barcode-Floer}. By
Theorem \ref{thm:two-defs} and in particular \eqref{eq:two-defs}, we
have
$$
\hbar(\alpha)=\lim_{\eps\to 0^+}\limsup_{s\to\infty}\frac{\log^+
  \fbdyn_\eps(sH)}{s}.
$$
Here we may assume that $s\not\in \CS(\alpha)$, i.e.,
$\slope(sH)\not\in\CS(\alpha)$ since $a=1$, and hence $\HF(sH)$ is
defined directly, without passing to the limit.

\medskip\noindent\emph{Step 2: Reduction to locally maximal sets.}
Recall that $K$ is a compact hyperbolic invariant set of the Reeb flow
$\varphi^t_\alpha$. The goal of this step is to show that without loss
of generality we can assume that $K$ is locally maximal in addition to
being hyperbolic. (Here we closely follow the first step in the proofs
of \cite[Thm.\ B]{GGM} and \cite[Thm.\ B]{CGG:Entropy}.)  Note first
that, by the variational principle for topological entropy,
\cite[Cor.\ 4.3.9]{FH}, for every $\delta>0$ there exists an invariant
probability measure $\mu$ supported in $K$ such that
$h_\mu\geq \htop(K)-\delta$, where $h_\mu$ is the metric entropy of
the Reeb flow $\varphi^t_\alpha|_K$. Moreover, the measure $\mu$ can
be chosen ergodic; see \cite[Thm.\ 8.4]{Wa}. Then, by \cite[Thm.\
D']{LY} extending \cite[Thm.\ S.5.9.(1)]{KH} to flows, whenever
$h_\mu>0$, there exists a locally maximal hyperbolic set $K'$ contained
in a neighborhood of $\supp \mu$, and hence in a neighborhood of $K$,
such that
$$
\htop(K')\geq h_\mu-\delta\geq \htop(K)-2\delta.
$$
Since $\delta>0$ is arbitrary, replacing $K$ by $K'$ we may assume
that $K$ is locally maximal and hyperbolic. Now Theorem \ref{thm:CE}
can be applied to $H$ and $K$, which we will do in the last step of
the proof.

\medskip\noindent\emph{Step 3: Periodic orbits in $K$ with action
  constraint.}  Denote by $p(s)$ the number of periodic orbits of the
flow $\varphi^t_\alpha|_K$ with period $T\leq s$. By \cite[Thm.\
5.4.22]{FH}, we have
$$
\htop(K)=\limsup_{s\to\infty}\frac{\log^+ p(s)}{s}=:L
$$
since $K$ is hyperbolic. Furthermore, fix $r_-$ in the range
$(1,\, r^+)$ and set $I'=(r_-,\, r_+]$. Recall that in Theorem
\ref{thm:CE}, $I=[r_-,\, r_+]$, and hence $I'\subset I$.

Let $p^H(s)$ be the number of $s$-periodic orbits $\tz=(z,r)$ of the
flow $\varphi_H^t$ with $z$ in $K$ and $r\in I'$. Equivalently,
$p^H(s)$ is the number of non-constant 1-periodic orbits $\tz$ of the
flow of $sH$ with $r\in I'$. The goal of this step is to show that
\begin{equation}
  \label{eq:orbit_count}
  h'(r_+)\cdot L \leq
  \limsup_{s\to\infty}\frac{\log^+ p^H(s)}{s}\leq L.
\end{equation}
In fact, we will just need the inequality on the left. The inequality
on the right, which is nearly obvious, is included only for the sake
of completeness.

Note that since $a=1$ and $X_H= h'(r)R$, where $R$ is the Reeb vector
field, there is a one-to-one correspondence between closed Reeb orbits
$z$ in $K$ with period $T\leq s$ and $s$-periodic orbits $\tz=(z,r)$
of the flow of $H$. In particular, $p^H(s)\leq p(s)$ and the second
inequality follows.

Let $\tz=(z,r^*)$ be an $s$-periodic orbit of the flow of $H$ and $T$
be the period of $z$ as an orbit of the Reeb flow. Then, as in
\eqref{eq:level}, $T$ and $r^*$ are related by the condition
$$
s h'(r^*)= T.
$$
Set
$$
a_-:=h'(r_-)\textrm{ and }a_+:=h'(r_+).
$$
Then, for $r^*$ to be in $I'=(r_-,\, r_+]$ we must have
$$
a_- s\leq T < a_+ s,
$$
and thus
\begin{equation}
  \label{eq:pH}
p^H(s)=p(a_+s)-p(a_-s).
\end{equation}
(The reason that we took $I'$ to be a semi-open interval rather than
the closed interval $I$ is to ensure that this equality holds
literally.)

Clearly,
\begin{equation*}
  \begin{split}
    p(a_+s) &=\big(p(a_+s) -p(a_-s) \big)+p(a_-s)\\
    & \leq \max\big\{2\big(p(a_+s) -p(a_-s) \big),\, 2p(a_-s)\big\}.
  \end{split}
\end{equation*}

Therefore,
\begin{equation*}
  \begin{split}
    a_+\cdot L & =  \limsup_{s\to\infty}\frac{\log^+
      p(a_+ s)}{s}\\
    & \leq 
    \limsup_{s\to\infty}\frac{1}{s}\log^+\max\big\{2\big(p(a_+s)
    -p(a_-s) \big),\, 2p(a_-s)\big\}\\
    & =
    \limsup_{s\to\infty}\frac{1}{s}\max\big\{\log^+\big(2\big(p(a_+s)
    -p(a_-s) \big)\big),\, \log^+\big(2p(a_-s)\big)\big\}\\
    & =
    \max\left\{\limsup_{s\to\infty}\frac{\log^+\big(2\big(p(a_+s)
    -p(a_-s) \big)\big)}{s},\,
  \limsup_{s\to\infty}\frac{\log^+\big(2p(a_-s)\big)}{s}\right\}\\
& =
    \max\left\{\limsup_{s\to\infty}\frac{\log^+p^H(s)}{s},\,
      a_- \cdot L\right\}.
   \end{split}
\end{equation*}
Here, in the last equality, we have used \eqref{eq:pH}. The second
term in the last line is strictly smaller than $a_+ L$ since
$a_-<a_+$, and hence the first term must be greater than on equal to
$a_+ L$. This proves the first inequality in
\eqref{eq:orbit_count}.

\medskip\noindent\emph{Step 4: Punchline.} Let, as in Remark
\ref{rmk:CE}, $(G,\tJ)$ be a compactly supported, regular
$C^\infty$-small perturbation of the pair $(sH,J)$. Under this
perturbation, every non-constant non-degenerate $s$-periodic orbit
$\tz=(z,r^*)$ of the flow of $H$ splits into at least two 1-periodic
orbits of $G$. Furthermore, by Theorem \ref{thm:CE} and again Remark
\ref{rmk:CE}, whenever $r^*\in I=[r_-,\, r_+]$ a Floer cylinder
asymptotic to any of these orbits of $G$ at either end has energy
greater than some constant $\sigma'>0$ which is independent of $s$ and
$\tz$. (However, the upper bound on the size $\|sH-G\|_{C^\infty}$ of
the perturbation may depend on $s$.)  Note also that $a_+s<s-\eps$
when $s$ is sufficiently large, and hence all such orbits have action
below $s-\eps$.

Assume now that $2\eps<\sigma'$. It follows from, e.g., \cite[Prop.\
3.8]{CGG:Entropy} that the Floer persistence module $\HF^\tau(G)$ has
at least $p^H(s)/2$ bars of length greater than $2\eps$. Since $sH$
and $G$ are $C^\infty$-close, the same is true for the Floer
persistence module $\HF^\tau(sH)$ with $2\eps$ replaced by
$\eps$. (Here, the perturbation $G$ is chosen after $\eps$ is
fixed.). In other words,
$$
\fbdyn_\eps(sH)\geq p^H(s)/2 .
$$

By \eqref{eq:two-defs} and \eqref{eq:orbit_count},
\begin{equation*}
  \begin{split}
    \hbar(\alpha) &\geq
\limsup_{s\to\infty}\frac{\log^+
  \fbdyn_\eps(sH)}{s}\\ &\geq
\limsup_{s\to\infty}\frac{\log^+
  p^H(s)}{s}\\ &\geq h'(r_+)\cdot \htop(K) \\ &\geq (1-\eta)\htop(K).
\end{split}
\end{equation*}
Thus
$$
\hbar(\alpha)\geq (1-\eta)\htop(K).
$$
As was pointed out in Step 1, we can take $\eta>0$ arbitrarily close
to 0. It follows that
$$
\hbar(\alpha)\geq \htop(K),
$$
which concludes the proof of Theorem \ref{thm:B}.
\hfill\qed

\section{Proof of the Crossing Energy Theorem}
\label{sec:cross}
The goal of this section is to prove the Crossing Energy Theorem --
Theorem \ref{thm:CE}. The proof follows the same general line of
reasoning as several other arguments of this type. The key new
ingredient which makes the proof work for Reeb flows is a location
constraint theorem from \cite{CGGM}.  We state this result in the next
section.

\subsection{Refinement and location constraints}
\label{sec:location}
We start this section with a refinement of Theorem \ref{thm:CE}, which
better reflects the logical structure of the proof.

In what follows, the Hamiltonian $H$ is assumed to be semi-admissible
and $J$ is admissible. In particular, $H$ and $J$ are independent of
time. We will also assume that all Floer cylinders $u$ we consider
have sufficiently small energy, and hence, by monotonicity, are
contained $M\times (1-\eta,\,\infty)$ for some small $\eta>0$; see
Section \ref{sec:Floer-eq}. In particular, the projection of $u$ to
$M$ is defined.

\begin{Theorem}[Crossing Energy Theorem, II]
  \label{thm:CE2}
  Let $K$ be a compact invariant set of the Reeb flow. Fix an
  admissible almost complex structure $J$ and an interval
  $$
  I=[r_-,\, r_+]\subset (1,\,\rmax).
  $$
  Let $H(r,x)=h(r)$ be a semi-admissible Hamiltonian with slope
  $a\not\in\CS(\alpha)$ such that \eqref{eq:h'''1} is satisfied for
  some $\delta>0$ with $r_++\delta<\rmax$. Let $\tau>0$ and
  $u \colon \R \times S^1 \to \WW$ be a Floer cylinder for $\tau H$
  asymptotic, at either end, to a 1-periodic orbit in $K\times I$.
\begin{itemize}
\item[\reflb{CE-P1}{\rm{(i)}}] Assume that $K$ is locally maximal and
  let $U$ be an isolating neighborhood of $K$. Then
  \begin{equation}
    \label{eq:CE2}
  E(u)\geq \sigma
\end{equation}
for some constant $\sigma>0$ independent of $\tau$ and $u$, whenever
$u$ is not entirely contained in
$\hU:=U\times [1,\,\infty)\subset \WW$.

\item[\reflb{CE-P2}{\rm{(ii)}}] Assume that $K$ is hyperbolic and
  $U\setminus K$ contains no periodic orbits of the Reeb flow for some
  neighborhood $U$ of $K$.  Then, when $U$ is sufficiently small and
  $u$ is entirely contained in $\hU$, the energy lower bound
  \eqref{eq:CE2} holds again with $\sigma>0$ independent of $\tau$ and
  $u$.
\end{itemize}
\end{Theorem}

Theorem \ref{thm:CE} readily follows from Theorem \ref{thm:CE2}, for
the requirements of both parts of the theorem are satisfied when $K$
is locally maximal and hyperbolic.  Part \ref{CE-P1} of Theorem
\ref{thm:CE2} is already sufficient for many purposes. For instance,
when $K$ comprises just one locally maximal periodic orbit, Part
\ref{CE-P2} is void and Part \ref{CE-P1} generalizes \cite[Thm.\
4.1]{CGGM}. For a fixed $\tau$, the lower bound, \eqref{eq:CE2}, is
again stable under small perturbations of $H$ and $J$ as in Remark
\ref{rmk:CE}. Note also that the parameter $s$ in $s H$ is renamed
here as $\tau$ since $s$ serves as the $\R$-coordinate in the domain
$\R\times S^1$ of a Floer cylinder $u$ in the proof of Theorem
\ref{thm:CE2}.

The central component of the proof of the Crossing Energy Theorem is
the following result which, under a minor additional condition on $H$,
along the lines of \eqref{eq:h'''1}, limits the range of $r\circ u$.

\begin{Theorem}[Location Constraint -- Thm.\ 6.1, \cite{CGGM}]
\label{thm:location}
Let $H(r,x) =h(r)$ be a semi-admissible Hamiltonian.  Assume that
$1<r_*^-\leq r_*^+$ and $\delta >0$ are such that
$$
1<r_*^--\delta \text{ and }
r_*^+ +\delta \leq \rmax,
$$
and
\begin{equation}
  \label{eq:h'''2}
  h''' \geq 0 \text{ on }
  [1, \, r_*^++\delta] \subset [1, \, \rmax).
\end{equation}
Fix an admissible almost complex structure $J$. Then there exists
$\sigma_0 >0$ such that for any $\tau >0 $ and any Floer cylinder
$u \colon \R \times S^1 \to \WW$ for $\tau H$ with energy
$E(u) \leq \sigma_0$ and asymptotic, at either end, to a periodic
orbit in $M\times [r_*^-,\,r_*^+]$, the image of $u$ is contained in
$M\times (r_*^--\delta,\, r_*^++\delta)$.
\end{Theorem}

In other words, a small energy Floer cylinder for $\tau H$ asymptotic
at either end to a periodic orbit in the shell
$M\times [r_*^-,\,r_*^+]$ must be entirely contained in a slightly
larger shell $M\times (r_*^--\delta,\, r_*^++\delta)$. The key
non-trivial part of this theorem is the lower bound $r_*^--\delta$ and
this is the part we will actually use. The upper bound is quite
standard and we included it only for the sake of completeness. We
refer the reader to \cite{CGGM} for the proof of the theorem.
(Strictly speaking Theorem \ref{thm:location} is proved there for
$\tau\in\N$. However, the argument carries over word-for-word to the
case of $\tau\in (0,\,\infty)$.)

\begin{Remark}
  As in Remark \ref{rmk:CE}, by the target compactness theorem from
  \cite{Fi}, the assertion of the theorem still holds with perhaps a
  smaller value of $\sigma_0$ when $sH$ and $J$ are replaced by their
  compactly supported, $\tau$-periodic in time $C^\infty$-small
  perturbations.  (The size of the perturbation may depend on $s$.) We
  also note that in Theorems \ref{thm:CE}, \ref{thm:CE2} and
  \ref{thm:location}, we could have required $H$ to be admissible
  rather than semi-admissible.
\end{Remark}  

\begin{Remark}
  The analogue of Theorem \ref{thm:CE} and Part \ref{CE-P2} of Theorem
  \ref{thm:CE2} for geodesic flows, \cite[Thm.\ D]{GGM}, holds when
  the hyperbolicity condition is replaced by a weaker requirement that
  $K$ is expansive; see, e.g., \cite[Def.\ 3.2.11]{KH} or \cite[Sec.\
  1.7 and 5.3]{FH} for the definition. It is conceivable that this is
  also true in the present setting although the proof from \cite{GGM}
  does not extend to general Reeb flows. However, as is easy to see
  from the proof below and \cite[Sec.\ 5.3]{FH}, in Part \ref{CE-P2}
  hyperbolicity can be replaced by expansivity and shadowing. See also
  \cite[Thm.\ 4.1]{CGGM}.
\end{Remark}  

\subsection{Energy bounds -- proof of Theorem \ref{thm:CE2}}
  \label{sec:CE2-pf}
  There are several sufficiently different approaches to the proof of
  crossing energy type results. Historically, the first one was based
  on target Gromov compactness, \cite{Fi}, and used in the proof of
  the original Crossing Energy Theorem in \cite{GG:hyperbolic} and
  then in \cite{GG:PR, CGG:Entropy} and more recently and in a more
  sophisticated form in \cite{CGP}. The second approach relies on the
  upper bound \eqref{eq:salamon} below from, e.g., \cite{Sa90, Sa},
  and on the conceptual level is also closely related to (the proof
  of) Gromov compactness. This method is pointed out in \cite[Rmk.\
  6.4]{CGG:Entropy} and then developed in \cite{CGGM}; see also
  \cite{Me:new}. The third approach is technically quite different. It
  uses finite-dimensional approximations and is fundamentally based on
  a ``classical'' form of Morse theory and the existence of the
  gradient flow; see \cite{Al, GGM}. This approach does not fit well
  in the general Floer theory setting, but either of the first two
  methods can be employed to prove Theorem \ref{thm:CE2}. Here,
  following \cite{CGGM}, we have chosen the second one, which is more
  hands-on and direct, albeit somewhat less general. The proof again
  comprises four steps.

  \medskip\noindent\emph{Step 1: From energy to $L^\infty$-upper
    bounds.}  Throughout the proof, it is convenient to adopt a
  different Hamiltonian iteration procedure. Namely, rather than
  looking at 1-periodic orbits of the Hamiltonian $\tau H$, we will
  look at the $\tau$-periodic orbits of $H$, changing the time range
  from $S^1=\R/\Z$ to $S^1_\tau=\R/\tau\Z$. We will refer to the
  resulting Hamiltonian as $H^{\sharp \tau}$; cf.\
  \cite{GG:hyperbolic, GG:PR}. This modification does not affect the
  Floer complexes, the action and the action filtration, the energy of
  a Floer trajectory, etc., with an isomorphism given by an
  $(s,t)$-reparametrization.

  With this in mind, by \cite[Sec.\ 1.5]{Sa} and \cite{Sa90}, there
  exist constants $C_H>0$ and $\sigma_H>0$ such that, for any Floer
  cylinder $u\colon \R\times S^1_\tau\to \WW$ for $H^{\sharp \tau}$
  with $E(u)<\sigma_H$, we have the point-wise upper bound
\begin{equation}
\label{eq:salamon}
\big\| \partial_s u(s, t) \big\| < C_H\cdot E(u)^{1/4},
\end{equation}
where $s$ is the coordinate on $\R$ and the norm on the left is
$L^\infty$; see also \cite{Br}.  The constants $\sigma_H$ and $C_H$
depend on $H$ via its first and second derivatives and $J$, but not on
$\tau$ or $u$.  (This is one instance where it is more convenient to
work with $H^{\sharp \tau}$ than $\tau H$; for the proof of
\eqref{eq:salamon} is local in the domain of $u$.)

Throughout the rest of the proof we will assume that $E(u)<\sigma_H$,
and hence $u$ satisfies \eqref{eq:salamon}.

\medskip\noindent\emph{Step 2: Pseudo-orbits.}  Recall that a map
$\gamma$ from an interval or a circle is an \emph{$\eta$-pseudo-orbit}
or just a \emph{pseudo-orbit} of the flow of $X$ if
$$
\big\|\dot{\gamma}(t)-X\big(\gamma(t)\big)\big\|<\eta
$$
for all $t$ in the domain of $\gamma$; see \cite[Def.\ 18.1.5]{KH} or
\cite{FH}. Then, when $\eta>0$ is small, $\gamma$ is close to the
integral curve of the flow through $x=\gamma(0)$. To be more precise,
as is easy to see, whenever a closed interval $\mathcal{I}$ in the
domain of $\gamma$ is fixed and $\eta>0$ is sufficiently small,
$\gamma|_{\mathcal{I}}$ is pointwise close to the integral curve of
the flow through $\gamma(0)$ on the same interval $\mathcal{I}$.

By \eqref{eq:salamon} and the Floer equation, \eqref{eq:floer_2},
Floer circles
$$
u(s)\colon t\mapsto u(s,t)
$$
are pseudo-orbits of the Hamiltonian flow of $H$ with
$\eta=C_H\cdot E(u)^{1/4}$, provided that $E(u)<\sigma_H$ as we have
assumed. Hence, when $E(u)$ is small, $u(s)$ approximates the integral
curve of the Hamiltonian flow through $u(0,0)$ over the interval
$[-\tau/2,\, \tau/2]$.

The goal of this step is to translate this fact from the Hamiltonian
flow of $H$ to the Reeb flow of $\alpha$ by suitably reparametrizing
$u(s)$.

Let $u$ be a Floer cylinder for $H^{\sharp \tau}$ with $E(u)<\sigma_H$
such that $\inf r(u)\geq \rmin$ for some constant $\rmin >1$. Note
that, by the maximum principle, $u$ is contained in the shell
$M\times [\rmin,\,\rmax]$. Set $\eps:=E(u)$ and $C:=C_H$. Then, by
\eqref{eq:salamon},
$$
\| \partial_s u-X_H \| < C\eps^{1/4}
$$
for all $s\in\R$, where $X_H=h' R_\alpha$. (This is another point
where it is more convenient to work with $H^{\sharp \tau}$ than with
$\tau H$; for $X_{H^{\sharp \tau}}=X_H$ is independent of $\tau$.)
Denote by $v$ the projection of $u$ to $M$. Then we also have
    \begin{equation}
      \label{eq:difference0}
  \big\| \p_t v - h'(r(u))R_\alpha \big\| < C\eps^{1/4}.
\end{equation}

Fixing $s\in\R$, let us reparametrize the map $t\mapsto v(s,t)$ by
using the change of variables $t=t(\xi)$ so that
$$ t'(\xi)= \frac{1}{h'\big(r(u(s,t))\big)},
  $$
  and set
  \begin{equation}
    \label{eq:gamma}
  \gamma(\xi)=\gamma_s(\xi):=v\big(s,t(\xi)\big).
\end{equation}
Then $\gamma$ is parametrized by the circle $S^1_T$ with 
\begin{equation}
  \label{eq:T}
    T\geq \tau \min_{t\in S^1_{\tau}}
    h'\big(r(u(0,t))\big) \geq \tau h'(\rmin).
\end{equation}

Furthermore, by \eqref{eq:difference0} and Theorem
\ref{thm:location},
\begin{equation}
  \label{eq:eta}
  \big\| \dot{\gamma}(\xi)- R_\alpha\big( \gamma(\xi) \big) \big\|
  < \frac{C\eps^{1/4}}{\min_{t\in S^1_\tau} h'\big(r(u(s,t))\big)}
  \leq \frac{C\eps^{1/4}}{h'(\rmin)}=:\eta,
\end{equation}
where the dot stands for the derivative with respect to $\xi$. (Note
that in this inequality we have used Theorem \ref{thm:location} in a
crucial way.) In other words, $\gamma$ is an $\eta$-pseudo-orbit of
the Reeb flow with $\eta$ completely determined by $\eps$ and other
auxiliary data, e.g., $H$, but independent of $s$ and $u$.

\medskip\noindent\emph{Step 3: Proof of Part \ref{CE-P1}.} Arguing by
contradiction, assume that there exists a sequence $\tau_k\to\infty$
and a sequence of Floer cylinders
$u_k\colon \R\times S^1_{\tau_k}\to \WW$ of $H^{\sharp \tau_k}$
satisfying the requirements of the theorem and such that
$$
\eps_k:= E(u_k)\to 0.
$$
In particular, $u_k$ is asymptotic to a $\tau_k$-periodic orbit
$\tz_k=(z_k,r_k^*)$ of the flow of $H$ where $z_k$ is a periodic orbit
of the Reeb flow in $K$ and $r^*_k\in I$. Moreover, $u_k$ is not
entirely contained in $\hU=U\times[1,\,\infty)$, where $U$ is an
isolating neighborhood of $K$. By shrinking $U$ if necessary, we can
guarantee that the closure $\bar{U}$ is a closed isolating
neighborhood, i.e, $K\subset U$ is a maximal invariant set in
$\bar{U}$, and that $u_k$ is not entirely contained in
$ \bar{U}\times [1,\, \infty)$.

Let $\delta>0$ be as in \eqref{eq:h'''1}. Setting $r_*^+=r_*^-=r_k^*$
in Theorem \ref{thm:location}, we have \eqref{eq:h'''2} satisfied.
Without loss of generality, we may assume that $\eps_k<\sigma_0$ and
hence that theorem applies to $u_k$. It follows that $u_k$ is
contained in the shell $M\times (r_--\delta,\, r_++\delta)$, where
$$
1<\rmin:=r_--\delta < r_++\delta<\rmax.
$$
Then $u_k$ is also contained in a bigger shell
$M\times [\rmin,\, \rmax]$ independent of $u_k$.

As in Steps 1 and 2, we assume that $\eps_k<\sigma_H$ and hence
\eqref{eq:salamon} holds. Finally, we may also require all $u_k$ to be
asymptotic to $\tz_k$ at the same end, say, $+\infty$. (For $-\infty$
the argument is identical.)

Suppressing $k$ in the notation, we set $u:=u_k$ and $\tau:=\tau_k$
and $\eps:=\eps_k$, etc.

The requirement that $u$ is not entirely contained in
$ \bar{U}\times [1,\, \infty)$ is equivalent to that $v$ is not
entirely contained in $\bar{U}$.  Since $u$ is asymptotic to $\tz$,
for some $s_0\in\R$ and $t_0\in S^1_{\tau}$, the curve
$v(s_0,S^1_{\tau})$ is contained in $\bar{U}$ while
$v(s_0,t_0)\in \p U:=\bar{U}\setminus U$. By translation and rotation
invariance of the Floer equation, without loss of generality, we may
assume that $s_0=0$ and $t_0=0$. Thus, $v(0,t)\in \bar{U}$ for all
$t\in S^1_{\tau}$ and $x:=v(0,0)\in \p U$.

Let $\gamma\colon S^1_T\to \bar{U}$ be defined by \eqref{eq:gamma}
with $s=s_0=0$ and $T$ satisfying \eqref{eq:T}. By \eqref{eq:eta},
$\gamma$ is an $\eta$-pseudo-orbit of the Reeb flow passing through
$x$.

Let us now reintroduce the subscript $k$ in the notation. To
summarize, we have found a sequence of $\eta_k$-pseudo-orbits
$\gamma_k\colon S^1_{T_k}\to \bar{U}$ of the Reeb flow with
$\eta_k\to 0$ by \eqref{eq:eta} and $T_k\to \infty$ by \eqref{eq:T}
and $x_k=\gamma_k(0)\in \p U$. We can view $\gamma_k$ as defined on
the interval $[-T_k/2,\, T_k/2]$ rather than on the circle
$S^1_{T_k}$. By passing if necessary to a subsequence and using the
diagonal process, we can ensure that $x_k\to x\in \p U$ and $\gamma_k$
converges uniformly on compact sets. Then the limit is the integral
curve $\xi\mapsto\varphi_\alpha^\xi(x)$, $\xi\in \R$, of the Reeb flow
passing through $x\in \p U\subset \bar{U}\setminus K$ and contained in
$\bar{U}$. This is impossible since $\bar{U}$ is a closed isolating
neighborhood of $K$.

\medskip\noindent\emph{Step 4: Proof of Part \ref{CE-P2}.} Let us fix
a neighborhood $U$ of $ K$ so small that the Shadowing Lemma applies
to pseudo-orbits in $U$; see, e.g., \cite[Thm.\ 18.1.]{KH} or
\cite[Sec.\ 5.3]{FH}.

As in Step 3, since $u$ is asymptotic at at least one end to a
periodic orbit in $K\times I$, it is entirely contained in
$U\times [\rmin,\, \rmax]$ whenever $\eps:=E(u)>0$ is sufficiently
small which we will require through the rest of the proof.  Thus $v$
is contained in $U$. Furthermore, $u$ is asymptotic to periodic orbits
$\tz_\pm=(z_\pm, \rho_\pm)$ of $H^{\sharp \tau}$ at $\pm\infty$ where
$z_\pm$ are periodic orbits of the Reeb flow in $K$ due to the
condition that all closed Reeb orbits of $\alpha$ in $U$ are contained
in $K$.  Let $T_\pm$ be the period of $z_\pm$.

Then, by \eqref{eq:Energy-Action},
$$
A_{\tau H}(\rho_+)-A_{\tau H}(\rho_-)=\eps>0,
$$
and hence $\rho_+>\rho_-$. At the same time, by \eqref{eq:level},
$$
\tau h'(\rho_\pm)= T_\pm.
$$
Therefore, $T_+>T_-$ because $h'$ is strictly increasing on $I$. It
follows that $z_\pm$, up to the initial condition, are distinct as
periodic orbits of the Reeb flow. In other words, we have an
alternative: either $z_\pm$ are geometrically distinct, i.e.,
$z_+(\R)\neq z_-(\R)$, or $z_+$ is a multiple cover of $z_-$.

Let $\gamma_s$ be as in \eqref{eq:gamma}. By \eqref{eq:eta},
$\gamma_s$ is an $\eta$-pseudo-orbit with
$\eta=C\eps^{1/4}/h'(\rmin)$. Furthermore, as $s\to \pm\infty$, the
loops $\gamma_s$ converge to $z_\pm$.  Due to the Shadowing Lemma,
when $\eps$ is sufficiently small, depending only on $U$ and the Reeb
flow but not $u$, for every $s$ there is a periodic orbit $\hgamma_s$
of the Reeb flow shadowing $\gamma_s$. This periodic orbit is unique,
depends continuously on $s$ and converges to $z_\pm$ as
$s\to\pm\infty$. In particular, the image $\Gamma_s:=\hgamma_s(\R)$
depends continuously on $s$ in Hausdorff topology and converges to
$z_+(\R)$ and $z_-(\R)$ as $s\to\pm\infty$.  Periodic orbits in $K$
are isolated and, therefore, we must have
$$
\Gamma_s=z_+(\R)=z_-(\R)
$$
for all $s$. Furthermore, $z_-$ is homotopic to $z_+$ in the circle
$z_+(\R)=z_-(\R)$. This is, however, impossible since at the same time
$z_+$ must then be a multiple cover of $z_-$ by the above alternative.

This contradiction concludes the proof of Theorem
\ref{thm:CE}. \hfill\qed


\begin{thebibliography}{CGG22b}

\bibitem{AASS}
  A. Abbondandolo, M.R.R. Alves, M. Sa\u{g}lam, F. Schlenk, Entropy
  collapse versus entropy rigidity for Reeb and Finsler flows,
  \emph{Selecta Math.\ (NS)} \textbf{29} (2023), Article No.\ 67.
  
\bibitem{AS}
  A. Abbondandolo, M. Schwarz, On the Floer homology of cotangent
  bundles, \emph{Comm.\ Pure Appl.\ Math.}, \textbf{59} (2006),
  254--316.  
    
\bibitem{Al}
  S. Allais, On periodic points of Hamiltonian diffeomorphisms of
  $\CP^d$ via generating functions, \emph{J. Symplectic Geom.},
  \textbf{20} (2022), 1–48. 

\bibitem{Al16}
  M.R.R. Alves, Cylindrical contact homology and topological entropy,
  \emph{Geom.\ Topol.}, \textbf{20} (2016), 3519--35692.

\bibitem{Al19}
  M.R.R. Alves, Legendrian contact homology and topological entropy,
  \emph{J. Topol.\ Anal.}, \textbf{11} (2019), 53-108. 

\bibitem{ACH}
  M.R.R. Alves, V. Colin, K. Honda, Topological entropy for Reeb
  vector fields in dimension three via open book decompositions,
  \emph{J.\ \'Ec.\ polytech.\ Math.}, \textbf{6} (2019), 119--148.

\bibitem{ADMM}
  M.R.R. Alves, L. Dahinden, M. Meiwes, L. Merlin, $C^0$-Robustness of
  topological entropy for geodesic flows, \emph{J.\ Fixed Point Theory
    Appl., Claude Viterbo’s 60th Birthday Festschrift.}, \textbf{24}
  (2022), doi:10.1007/s11784-022-00959-4.

\bibitem{ADMP}
  M.R.R. Alves, L. Dahinden, M. Meiwes, A. Pirnapasov, $C^0$-stability
  of topological entropy for Reeb flows in dimension 3, Preprint
  arXiv:2311.12001.
  
  
\bibitem{AM}
  M.R.R. Alves, M. Meiwes, Dynamically exotic contact spheres in
  dimensions $\geq 7$, \emph{Comment.\ Math.\ Helv.}, \textbf{94}
  (2019), 569--622. 
  
\bibitem{AP}
  M.R.R. Alves, A. Pirnapasov, Reeb orbits that force topological
  entropy, \emph{Ergodic Theory Dynam.\ Systems}, \textbf{42} (2022),
  3025--3068. 
  
  
\bibitem{As}  
  M. Asaoka, Abundance of fast growth of the number of periodic points
  in 2-dimensional area-preserving dynamics, \emph{Comm.\ Math.\
    Phys.}, \textbf{356} (2017), 1--17. 

\bibitem{ACW}  
  A. Avila, S. Crovisier, A. Wilkinson, $C^1$ density of stable
  ergodicity, \emph{Adv.\ Math.}, \textbf{379} (2021), Paper No.\
  107496, 68 pp.

\bibitem{BG}
  E. Barut, V.L. Ginzburg, Barcode growth for toric-integrable
  Hamiltonian systems, Preprint arXiv:2503.08922.

  
\bibitem{Ba1}  
  M. Bator\'eo, On hyperbolic points and periodic orbits of
  symplectomorphisms, \emph{J. Lond.\ Math.\ Soc.\ (2)}, \textbf{91}
  (2015), 249--265. 

 \bibitem{Ba2} 
   M. Bator\'eo On non-contractible hyperbolic periodic orbits and
   periodic points of symplectomorphisms, \emph{J. Symplectic Geom.},
   \textbf{15} (2017), 687--717.  

\bibitem{Bo}   
  F. Bourgeois, A Morse--Bott approach to contact homology, In
  \emph{Symplectic and contact topology: interactions and perspectives
    (Toronto, ON/Montreal, QC, 2001)}, Fields Inst.\ Commum., AMS,
  \textbf{35}, (2003), 55--77.

\bibitem{Br}  
  B. Bramham, Pseudo-rotations with sufficiently Liouvillean rotation
  number are $C^0$-rigid, \emph{Invent.\ Math.}, \textbf{199} (2015),
  561--580.  

\bibitem{BV}  
  P. Bubenik, T.Vergili, Topological spaces of persistence modules and
  their properties, \emph{J. Appl.\ Comput.\ Topol.}, \textbf{2}
  (2018), 233--269. 

\bibitem{CZCG}  
  G. Carlsson, A. Zomorodian, A. Collins, L. Guibas, Persistence
  barcodes for shapes, \emph{Int.\ J. Shape Model.}, \textbf{11}
  (2005), 149--187.

\bibitem{CFH}  
  K. Cieliebak, A. Floer, H. Hofer, Symplectic homology II: A general
  construction, \emph{Math.\ Zeit.}, \textbf{218} (1995), 103--122.

\bibitem{Ci:counterexample}  
  E. \c Cineli, A generalized pseudo-rotation with positive
  topological entropy, \emph{Bull.\ Lond.\ Math.\ Soc.}, 2025, doi:
  10.1112/blms.70021.

  
\bibitem{CGG:Entropy}  
  E. \c Cineli, V.L. Ginzburg, B.Z. G\"urel, Topological entropy of
  Hamiltonian diffeomorphisms: a persistence homology and Floer theory
  perspective, 
  \emph{Math.\ Z.}, 
  \textbf{73} (2024), doi: 10.1007/s00209-024-03627-0.
  
\bibitem{CGG:Growth}  
  E. \c Cineli, V.L. Ginzburg, B.Z. G\"urel, On the growth of the
  Floer barcode, 
  \emph{J.\ Mod.\ Dyn.}, \textbf{20} (2024), 275--298, doi:
  10.3934/jmd.2024007.
  

\bibitem{CGG:Spectral}  
  E. \c Cineli, V.L. Ginzburg, B.Z. G\"urel, On the generic behavior
  of the spectral norm, 
  \emph{Pacific J.\ Math.}, \textbf{328} (2024), no.\ 1, 119--135,
  doi: 10.2140/pjm.2024.328-119.


\bibitem{CGG:ReebHZ}
  E. \c Cineli, V.L. Ginzburg, B.Z. G\"urel, Closed orbits of
  dynamically convex Reeb flows: Towards the HZ- and multiplicity
  conjectures, Preprint arXiv:2410.13093.
  

\bibitem{CGGM}  
  E. \c Cineli, V.L. Ginzburg, B.Z. G\"urel, M. Mazzucchelli,
  Invariant sets and hyperbolic closed Reeb orbits, Preprint
  arXiv:2309.04576.

  
\bibitem{CB}  
  W. Crawley-Boevey, Decomposition of pointwise finite-dimensional
  persistence modules, \emph{J. Algebra Appl.}, \textbf{14} (2015),
  no.\ 5, 1550066, 8 pp.

\bibitem{CGP}  
  D. Cristofaro-Gardiner, R. Prasad, Low-action holomorphic curves and
  invariant sets, Preprint arXiv:2401.14445.
  
\bibitem{Di}  
  E.I. Dinaburg, A connection between various entropy
  characterizations of dynamical systems, \emph{Izv.\ Akad.\ Nauk SSSR
    Ser.\ Mat.}, \textbf{35} (1971), 324--366. 

\bibitem{FLS}
  E. Fender, S. Lee, B. Sohn, Barcode entropy for Reeb flows on
  contact manifolds with exact Liouville fillings, \emph{Comm.\
    Contemp.\ Math.}, 2025, doi.org/10.1142/S0219199725500440.

\bibitem{Fe24} 
  R. Fernandes, Barcode entropy and wrapped Floer homology, Preprint
  arXiv:2410.05528.

\bibitem{Fe25} 
  R. Fernandes, Wrapped Floer homology and hyperbolic sets, Preprint
  arXiv:2501.06654.
  
\bibitem{Fi}  
  J.W. Fish, Target-local Gromov compactness, \emph{Geom.\ Topol.},
  \textbf{15} (2011), 765--826.

\bibitem{FH}  
  T. Fisher, B. Hasselblatt, \emph{Hyperbolic Flows}, Zurich Lectures
  in Advanced Mathematics, European Mathematical Society, Berlin,
  2019.

\bibitem{FS}  
  U. Frauenfelder, F. Schlenk, Hamiltonian dynamics on convex
  symplectic manifolds, \emph{Israel J.\ Math.}, \textbf{15} (2006),
  1--56.

\bibitem{GG:hyperbolic}  
  V.L. Ginzburg, B.Z. G\"urel, Hyperbolic fixed points and periodic
  orbits of Hamiltonian diffeomorphisms, \emph{Duke Math. J.},
  \textbf{163} (2014), 565--590.

\bibitem{GG:nc}  
  V.L. Ginzburg, B.Z. G\"urel, Non-contractible periodic orbits in
  Hamiltonian dynamics on closed symplectic manifolds, \emph{Compos.\
    Math.}, \textbf{152} (2016), 1777--1799.

\bibitem{GG:PR}  
  V.L. Ginzburg, B.Z. G\"urel, Hamiltonian pseudo-rotations of
  projective spaces, \emph{Invent.\ Math.}, \textbf{214} (2018),
  1081--1130.

\bibitem{GG:LS}  
  V.L. Ginzburg, B.Z. G\"urel, Lusternik--Schnirelmann theory and
  closed Reeb orbits, \emph{Math.\ Z.}, \textbf{295} (2020),
  515--582. 

\bibitem{GGM}  
  V.L. Ginzburg, B.Z. G\"urel, M. Mazzucchelli, Barcode entropy of
  geodesic flows, 
  \emph{J.\ Eur.\ Math.\ Soc. (JEMS)}, doi: 10.4171/JEMS/1572.

  
\bibitem{Kal}  
  V. Kaloshin, An extension of the Artin-Mazur theorem, \emph{Ann.\ of
    Math.} (2), \textbf{150} (1999), 729--741. 

\bibitem{Ka}  
  A. Katok, Lyapunov exponents, entropy and periodic orbits for
  diffeomorphisms, \emph{Inst.\ Hautes \'Etudes Sci.\ Publ.\ Math.},
  \textbf{51} (1980), 137--173.

\bibitem{Ka82}
  A. Katok, Entropy and closed geodesics, \emph{Ergodic Theory Dynam.\
    Systems}, \textbf{2} (1982), 339--365. 

\bibitem{KH}
  A. Katok, B. Hasselblatt, \emph{Introduction to the Modern Theory of
    Dynamical Systems.} With a supplementary chapter by A. Katok and
  Mendoza. Encyclopedia of Mathematics and its Applications,
  54. Cambridge University Press, Cambridge, 1995.

\bibitem{LY}  
  Z. Lian, L.-S. Young, Lyapunov exponents, periodic orbits, and
  horseshoes for semiflows on Hilbert spaces, \emph{J. Amer.\ Math.\
    Soc.}, \textbf{25} (2012), 637--665.  

\bibitem{LS}   
  Y. Lima, O.M. Sarig, Symbolic dynamics for three-dimensional flows
  with positive topological entropy,\emph{ J. Eur.\ Math.\ Soc.\
    (JEMS)}, \textbf{21} (2019), 199--256. 

\bibitem{MS}  
  L. Macarini, F. Schlenk, Positive topological entropy of Reeb flows
  on spherizations, \emph{Math.\ Proc.\ Cambridge Philos.\ Soc.},
  \textbf{151} (2011), 103--128. 

\bibitem{Me}  
  M. Meiwes, \emph{Rabinowitz Floer Homology, Leafwise Intersections,
    and Topological Entropy}, Inaugural Dissertation zur Erlangung der
  Doktorw\"urde der Naturwissenschaftlich -- Mathematischen
  Gesamtfakultät der Ruprecht -- Karls -- Universit\"at Heidelberg,
  2018; available at http://www.ub.uni-heidelberg.de/archiv/24153.

\bibitem{Me:new}  
  M. Meiwes, On the barcode entropy of Lagrangian submanifolds,
  Preprint arXiv:2401.07034.

\bibitem{Pa}   
  G.P. Paternain, \emph{Geodesic Flows}, Progress in Mathematics,
  vol.\ 180, Birkha\"user Boston, Inc., Boston, MA, 1999. 

\bibitem{PRSZ}  
  \emph{Topological Persistence in Geometry and Analysis}, University
  Lecture Series, vol.\ 74, Amer.\ Math.\ Soc., Providence, RI, 2020.

\bibitem{Sa90}  
  D.A. Salamon, Morse theory, the Conley index and Floer homology,
  \emph{Bull.\ Lond.\ Math.\ Soc.}, \textbf{22} (1990), 113--140.

\bibitem{Sa}  
  D.A. Salamon, Lectures on Floer homology. In \emph{Symplectic
    Geometry and Topology}, IAS/Park City Math.\ Ser., vol.\ 7, Amer.\
  Math.\ Soc., Providence, RI, 1999, 143--229.

\bibitem{SW}  
  D.A. Salamon, J. Weber, Floer homology and the heat flow,
  \emph{Geom.\ Funct.\ Anal.}, \textbf{16} (2006),
  1050--1138. 

\bibitem{Si}  
  J.-C. Sikorav, Some properties of holomorphic curves in almost
  complex manifolds, in \emph{Holomorphic curves in symplectic
    geometry}, pp. 165--189; Progr.\ Math., \textbf{117} Birkh\"auser
  Verlag, Basel, 1994.

\bibitem{Vi}  
  C. Viterbo, Functors and computations in Floer cohomology, I,
  \emph{Geom.\ Funct.\ Anal.}, \textbf{9} (1999), 985--1033.

\bibitem{Wa}  
  P. Walters, \emph{An Introduction to Ergodic Theory}, Graduate Texts
  in Mathematics, vol.\ 79, Springer-Verlag, New York-Berlin, 1982.

\bibitem{ZC}  
  A. Zomorodian, G. Carlsson, Computing persistent homology,
  \emph{Discrete Comput.\ Geom.}, \textbf{33} (2005), 249--274. 

  
\end{thebibliography}
\end{document}